\documentclass[reqno]{amsart}

\usepackage{amsmath}
\usepackage{amsthm}
\usepackage{amssymb}

\usepackage{mathrsfs} 
\usepackage{mathtools}

\usepackage{bm}
\usepackage{courier}
\usepackage{fnpct}
\usepackage{stmaryrd}
\usepackage[dvipsnames]{xcolor}
\usepackage{subcaption}
\usepackage{stmaryrd}
\usepackage{hyperref}
    \hypersetup{
        colorlinks,
        linkcolor={blue!50!white},
        citecolor={blue!50!white},
        urlcolor={blue!80!white}
    }
\usepackage{comment}

\usepackage{tikz}
\usepackage{pgfplots}
\usetikzlibrary{arrows.meta}
\usetikzlibrary{shapes.misc}
\usepackage{ifthen}
\usepackage{tkz-graph}

\pgfdeclarelayer{background}
\pgfdeclarelayer{foreground}
\pgfsetlayers{background,main,foreground}

\numberwithin{equation}{section}

\newcommand{\R}{\mathbb{R}}

\newcommand{\Z}{\mathbb{Z}}

\newcommand{\C}{\mathcal{C}}
\newcommand{\D}{\mathcal{D}}
\newcommand{\LWM}{\mathrm{LWM}}
\newcommand{\WM}{\mathrm{WM}}
\newcommand{\M}{\mathrm{M}}

\newcommand{\restrict}{\upharpoonright}

\newcommand{\diam}{\mathrm{diam}}

\newcommand{\type}{\mathrm{type}}

\newcommand{\maj}{\prec}

\newcommand{\floor}[1]{\left\lfloor #1\right\rfloor}
\newcommand{\ceil}[1]{\left\lceil #1\right\rceil}

\renewcommand{\emptyset}{\varnothing}
\newcommand{\st}{\mid}

\newtheorem{theorem}{Theorem}[section]
\newtheorem{lemma}[theorem]{Lemma}
\newtheorem{observation}[theorem]{Observation}
\newtheorem{corollary}[theorem]{Corollary}

\theoremstyle{definition}

\newtheorem{question}[theorem]{Question}
\newtheorem{problem}[theorem]{Problem}

\newtheorem{definition}[theorem]{Definition}
\newtheorem{example}[theorem]{Example}

\author[V. ~Bardenova]{Viktoriya Bardenova}
\address[Viktoriya Bardenova]{ 
Virginia Commonwealth University,
Department of Mathematics and Applied Mathematics,
1015 Floyd Avenue, PO Box 842014, Richmond, Virginia 23284, United States
} 
\email[V. ~Bardenova]{bardenovav@vcu.edu} 

\author[N. ~Bushaw]{Neal Bushaw}
\address[Neal Bushaw]{ 
Virginia Commonwealth University,
Department of Mathematics and Applied Mathematics,
1015 Floyd Avenue, PO Box 842014, Richmond, Virginia 23284, United States
} 
\email[N. ~Bushaw]{nobushaw@vcu.edu} 
\urladdr{https://thenealon.github.io/}

\author[B. ~Cody]{Brent Cody}
\address[Brent Cody]{ 
Virginia Commonwealth University,
Department of Mathematics and Applied Mathematics,
1015 Floyd Avenue, PO Box 842014, Richmond, Virginia 23284, United States
} 
\email[B. ~Cody]{bmcody@vcu.edu}
\urladdr{https://brentcody.github.io/}

\author[P. ~Fay]{Paul Fay}
\address[Paul Fay]{ 
Virginia Commonwealth University,
Department of Mathematics and Applied Mathematics,
1015 Floyd Avenue, PO Box 842014, Richmond, Virginia 23284, United States
} 
\email[P. ~Fay]{fayph@vcu.edu}

\author[M. ~Tennant]{Maya Tennant}
\address[Maya Tennant]{ 
Virginia Commonwealth University,
Department of Mathematics and Applied Mathematics,
1015 Floyd Avenue, PO Box 842014, Richmond, Virginia 23284, United States
} 
\email[M. ~Tennant]{tennantm@vcu.edu}

\thanks{All five authors were partially supported by a Seed Award from Virginia Commonwealth University. Special thanks to Sean Cox for partially supporting the first, fourth and fifth authors on National Science Foundation grant DMS-2154141.}

\date{\today}

\pgfplotsset{compat=1.18}
\begin{document}

\title{The Wiener index of vertex colorings}

\begin{abstract}
The Wiener index of a vertex coloring of a graph is defined to be the sum of all pairwise geodesic distances between vertices of the same color. We provide characterizations of vertex colorings of paths and cycles whose Wiener index is as large as possible over various natural collections. Along the way we establish a connection between the majorization order on tuples of integers and the Wiener index of vertex colorings on paths and cycles.
\end{abstract}

\subjclass[2020]{05C12, 05C15, 05C35}

\keywords{}

\maketitle




%


\section{Introduction}\label{section_introduction}



Vertex colorings whose color classes are ``spread out,'' in various senses, have been widely studied throughout graph theory. For example, a \emph{proper coloring} \cite{MR1633290, MR2368647} has color classes that are independent sets, a \emph{domatic coloring} \cite{MR584887} has color classes that are dominating sets, and a \emph{packing coloring} \cite{MR2360659, sh2023} has the property that its $i$th-color class is an $i$-packing. Furthermore, each of these topics has expanded to include variations with extra conditions on the color classes, or conditions on the pairwise relationships between color classes (see, e.g., \cite{MR2387114, MR593902}).  In this article, we study vertex colorings whose color classes are spread out in another sense.


Suppose $G$ is a finite simple connected graph. For vertices $u,v\in V(G)$ the \emph{distance from $u$ to $v$}, denoted by $d_G(u,v)$ is the length of a shortest path from $u$ to $v$ in $G$. For a set of vertices $A\subseteq V(G)$, the \emph{Wiener index of $A$ in $G$}, previously studied in \cite{BCL}, is the quantity
\[W_G(A)=\sum_{\{u,v\}\in \binom{A}{2}}d_G(u,v).\]  
We say that a set of vertices $A$ is a \emph{maximizer of $W$ on $G$} if 
\[W_G(A)=\max\{W_G(A')\st A'\subseteq V(G)\text{ and }|A'|=|A|\}.\]
Let us note that, although the Wiener index of sets of vertices has only been introduced recently, the Wiener index $W(V(G))$ of a graph $G$, an example of what is referred to as a \emph{topological invariant}, has been widely studied \cite{MR1843259,MR1916949,MR1461039,MR3570470} in the context of mathematical chemistry.

For sets of vertices $A$ and $B$, with $|A|=|B|$, when $W_G(A)>W_G(B)$, the pairwise distances between vertices of $A$ are, on average, larger than the pairwise distances between vertices of $B$. Thus, in a sense, maximizers of $W$ on $G$ are sets of vertices that are spread out as much as possible within $G$. Several characterizations of sets of vertices that maximize $W$ on paths and cycles were given in \cite{BCL}. In this article, we extend the notion of the Wiener index from sets of vertices to vertex colorings, and we characterize vertex colorings with maximal Wiener index on paths and cycles.

For a finite graph $G$ and a vertex $v\in V(G)$, we define $d_G(v)$ to be the tuple of distances from $v$ to all other vertices of $G$, arranged in non-decreasing order. If a graph $G$ is \emph{distance degree regular}, meaning that whenever $u,v\in V(G)$ we have $d_G(u)=d_G(v)$, then it follows \cite[Theorem 3.5]{BCL} that the complement of any maximizer of $W$ in $G$ is also a maximizer of $W$ in $G$. Therefore, if $G$ is distance degree regular, then a set of vertices $A$ is a maximizer of $W$ on $G$ if and only if the quantity $W(A)+W(V(G)\setminus A)$ is maximal among the values of $W(A')+W(V(G)\setminus A')$ for all $A'\subseteq V(G)$ with $|A'|=|A|$. Thus, we see that for a distance degree regular graph $G$, the problem of finding a set of vertices of a given cardinality $n_1$ that is a maximizer of $W$, is equivalent to that of finding a $2$-partition $\{V_1,V_2\}$ of $V(G)$ such that $|V_1|=n_1$ and $W(V_1)+W(V_2)$ is maximal among all quantities of the form $W(V_1')+W(V_2')$ where $\{V_1',V_2'\}$ is a $2$-partition of $V(G)$ with $|V_1'|=n_1$.


For a positive integer $k$, a \emph{$k$-coloring} of a graph $G$ is simply a function $f:V(G)\to\{1,\ldots,k\}$.\footnote{Let us note that the standard notion of \emph{proper vertex coloring}, in which adjacent vertices receive distinct colors, does not play a role in this article.} For each $i\in\{1,\ldots,k\}$ we let $V_i=f^{-1}(i)$, so that $\{V_1,\ldots,V_k\}$ is a $k$-partition of $V(G)$. The \emph{type of $f$} is the tuple of cardinalities of color classes of $f$ arranged in non-decreasing order; that is,
\[\type(f)=(|V_{i_1}|,\ldots,|V_{i_k}|)\]
where $|V_{i_1}|\leq\cdots\leq|V_{i_k}|$.
The \emph{Wiener index of $f$} is 
\[W(f)=\sum_{i=1}^kW(V_i).\]

If $f:V(G)\to\{1,\ldots,k\}$ is a function, and $u,v\in V(G)$, we let $S_{u,v}(f):V(G)\to\{1,\ldots,k\}$ denote the function which equals $f$ on $V(G)\setminus \{u,v\}$ and satisfies $S_{u,v}(f)(u)=f(v)$ and $S_{u,v}(f)(v)=f(u)$; that is, the colors of $u$ and $v$ are swapped. \begin{definition} \label{definition_maximizer}
Suppose $k\geq 2$ is an integer and $f:V(G)\to\{1,\ldots,k\}$ is a function. We say that
\begin{enumerate}
\item $f$ is a \emph{$k$-color local weak maximizer of $W$ on $G$} if $f$ is surjective and $W(f)\geq W(S_{u,v}(f))$ for all $\{u,v\}\in E(G)$,
\item $f$ is a \emph{$k$-color weak maximizer of $W$ on $G$} if $f$ is surjective and $W(f)\geq W(f')$ for all surjective functions $f':V(G)\to\{1,\ldots,k\}$ with $\type(f')=\type(f)$ and
\item $f$ is a \emph{$k$-color maximizer of $W$ on $G$} if $f$ is surjective and $W(f)\geq W(f')$ for all surjective functions $f':V(G)\to\{1,\ldots,k\}$.
\end{enumerate}
We let $\LWM_k(G)$, $\WM_k(G)$ and $\M_k(G)$ denote the collections of $k$-color local weak maximizers of $W$ on $G$, $k$-color weak maximizers of $W$ on $G$ and $k$-color maximizers of $W$ on $G$, respectively.
\end{definition}

Clearly, $\M_k(G)\subseteq\WM_k(G)\subseteq\LWM_k(G)$, but the reverse inclusions do not hold in general.

We let $C_n$ denote a cycle on $n$ vertices and $P_n$ a path on $n$ vertices. Several characterizations of sets of vertices in $C_n$ that are maximizers of $W$ were previously established in \cite{BCL}. In Section \ref{section_weak_max_cycles}, we review one such characterization and introduce certain canonical sets of vertices of $C_n$ that are maximizers of $W$, which we refer to as \emph{good} sets. Using the notion of good set, we prove that a surjective function $f:V(C_n)\to\{1,\ldots,k\}$ is a $k$-color weak maximizer of $W$ on $C_n$ if and only if for all $i\in\{1,\ldots,k\}$ the set $f^{-1}(i)$ is a maximizer of $W$ on $C_n$.

In Section \ref{section_local_weak_max_paths}, we prove that a surjective function $f:V(P_n)\to\{1,\ldots,k\}$ is a $k$-color weak maximizer of $W$ on $P_n$ if and only if it is $k$-color \emph{local} weak maximizer of $W$ on $P_n$, and we provide a complete specification of such functions; let us emphasize that the color classes of weak maximizers of $W$ on $P_n$ need not themselves be maximizers of $W$ as sets of vertices. Sets of vertices of $P_n$ that maximize $W$ were previously characterized in \cite{BCL}. 


In Section \ref{section_majorization}, we establish a connection between the Wiener index of $k$-color weak maximizers on paths and cycles and the majorization order $\maj$ (Definition \ref{definition_maj}), which originated in the work of Schur \cite{Schur1923} as well as Hardy, Littlewood and Polya \cite{HLP}, and which has also been applied in graph theory \cite{BCL, MR2364586, MR262112}.  

Using the results of Section \ref{section_local_weak_max_paths} characterizing the $k$-color weak maximizers on paths, we prove that given two $k$-color weak maximizers $f$ and $f'$ on $P_n$, $\type(f)\maj\type(f')$ and $\type(f)\neq \type(f')$ implies $W(f)<W(f')$ (Theorem \ref{theorem_maj_paths}). As a consequence (see Corollary \ref{corollary_largest_and_smallest_paths}), we deduce two things for paths: (1) the $k$-color maximizers of $W$ on $P_n$ are precisely the $k$-color weak maximizers of $W$ which have a color class of the largest possible size $n-(k-1)$ and (2) the $k$-color weak maximizers of $W$ on $P_n$ with the smallest possible Wiener index are precisely those whose color classes differ in size by at most one. In particular, by (1), this shows that the problem of finding $k$-color maximizers of $W$ on $P_n$ reduces to that of finding sets of vertices of $P_n$ that maximize the Wiener index. So, in a sense, at least for paths, $k$-color weak maximizers are more interesting than $k$-color maximizers. However, the situation for cycles is more complicated.

We discuss an example (Example \ref{example_bad}) which shows that the analogous results for cycles do not hold in general. For example, it is not true in general that the $k$-color maximizers of $W$ on $C_n$ are precisely the $k$-color weak maximizers with a color class of the largest possible size $n-(k-1)$. We then isolate the problematic cases and prove a modified version of this result for cycles (see Theorem \ref{theorem_maj_cycles}, Corollary \ref{corollary_W_does_not_go_down} and Figure \ref{figure_3cwm_maj}), which states that if $f$ and $f'$ are two $k$-color weak maximizers of $W$ on $C_n$, $\type(f)\maj\type(f')$ and $\type(f)\neq \type(f')$ then $W(f)\leq W(f')$, and furthermore, our results specify the conditions under which one can have $\type(f)\maj \type(f')$, $\type(f)\neq\type(f')$ and $W(f)=W(f')$ (on cycles). We then characterize which $k$-color weak maximizers of $W$ on $C_n$ are maximizers of $W$ on $C_n$ (see Corollary \ref{corollary_cycles_largest}), and which $k$-color weak maximizers of $W$ on $C_n$ have the smallest possible Wiener index (see Corollary \ref{corollary_cycles_smallest}).

In Section \ref{section_questions}, we state three broad problems as well as several specific questions relevant to the present work.

\section{Weak maximizers on cycles}\label{section_weak_max_cycles}

In this section we will prove that a surjective function $f:V(C_n)\to\{1,\ldots,k\}$ is a $k$-color weak maximizer of $W$ on $C_n$ if and only if all of the color classes of $f$ are weak maximizers of $W$ on $C_n$, as sets of vertices (see Figure \ref{figure_3cwm} for a depiction of the $3$-color weak maximizers of $W$ on $C_7$). Let us recall a characterization of sets of vertices that are maximizers of $W$ on $C_n$ established in \cite{BCL}. We say that a set of vertices $X$ of a finite graph $G$ is \emph{connected} if the induced subgraph $G[X]$ is connected. A partition $\mathcal{D}$ of the vertices of a graph is called \emph{equitable} if for all blocks $D_1,D_2\in\mathcal{D}$ the cardinalities $|D_1|$ and $|D_2|$ differ by at most one.

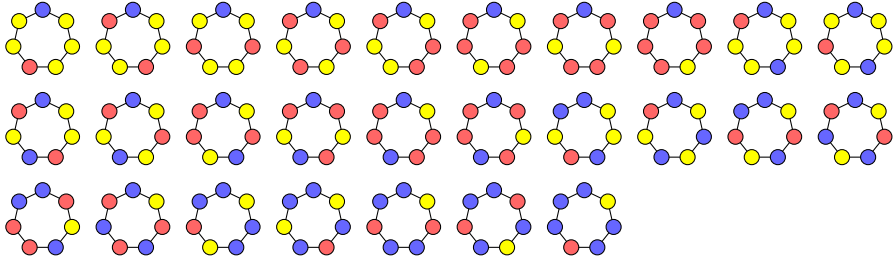
\begin{figure}
\centering
\begin{tikzpicture}[scale=0.4]
    \definecolor{r}{rgb}{1, 0.4, 0.4}    
    \definecolor{b}{rgb}{0.4, 0.4, 1}    
    \definecolor{y}{rgb}{1, 1, 0}
  
  \tikzset{
    ynode/.style={draw, fill=y, circle, minimum size=2mm, inner sep=0pt},
    rnode/.style={draw, fill=r, circle, minimum size=2mm, inner sep=0pt},
    bnode/.style={draw, fill=b, circle, minimum size=2mm, inner sep=0pt}
  }

  \newcommand{\cycle}[9]{
    \foreach \i in {1,...,7} {
      \coordinate (p\i) at ({#1 + cos((360/7)*(\i-1)+90)}, {#2 + sin((360/7)*(\i-1)+90)});
    }

    \foreach \i [count=\j] in {1,...,7} {
      \ifnum\j=7
        \draw (p\j) -- (p1);
      \else
        \draw (p\j) -- (p\the\numexpr\j+1\relax);
      \fi
    }

    \node[#3node] at (p1) {};
    \node[#4node] at (p2) {};
    \node[#5node] at (p3) {};
    \node[#6node] at (p4) {};
    \node[#7node] at (p5) {};
    \node[#8node] at (p6) {};
    \node[#9node] at (p7) {};

  }
  
  \cycle{0}{0}{b}{y}{y}{r}{y}{y}{y}
  \cycle{3}{0}{b}{r}{y}{y}{r}{y}{y}
  \cycle{6}{0}{b}{y}{r}{y}{y}{r}{y}
  \cycle{9}{0}{b}{r}{y}{r}{y}{r}{y}
  \cycle{12}{0}{b}{r}{y}{y}{r}{r}{y}
  \cycle{15}{0}{b}{r}{r}{y}{r}{r}{y}
  \cycle{18}{0}{b}{r}{y}{r}{r}{y}{r}
  \cycle{21}{0}{b}{r}{r}{r}{y}{r}{r}
  \cycle{24}{0}{b}{r}{y}{y}{b}{y}{y}
  \cycle{27}{0}{b}{y}{r}{y}{b}{y}{y}

  \cycle{0}{-3}{b}{r}{y}{b}{r}{y}{y}
  \cycle{3}{-3}{b}{r}{y}{b}{y}{r}{y}
  \cycle{6}{-3}{b}{r}{r}{y}{b}{r}{y}
  \cycle{9}{-3}{b}{r}{y}{b}{r}{y}{r}
  \cycle{12}{-3}{b}{r}{r}{b}{r}{r}{y}
  \cycle{15}{-3}{b}{r}{r}{b}{r}{y}{r}
  \cycle{18}{-3}{b}{b}{y}{r}{b}{y}{y}
  \cycle{21}{-3}{b}{r}{y}{b}{y}{b}{y}
  \cycle{24}{-3}{b}{b}{r}{y}{b}{r}{y}
  \cycle{27}{-3}{b}{r}{b}{y}{b}{r}{y}

  \cycle{0}{-6}{b}{b}{r}{r}{b}{y}{r}
  \cycle{3}{-6}{b}{r}{b}{r}{b}{r}{y}
  \cycle{6}{-6}{b}{b}{r}{y}{b}{b}{y}
  \cycle{9}{-6}{b}{b}{y}{b}{r}{b}{y}
  \cycle{12}{-6}{b}{b}{r}{b}{b}{r}{y}
  \cycle{15}{-6}{b}{b}{r}{b}{y}{b}{r}
  \cycle{18}{-6}{b}{b}{b}{r}{b}{b}{y}

\end{tikzpicture}
\caption{\small The $3$-color weak maximizers of $W$ on $C_7$ up to automorphism.}
\label{figure_3cwm}
\end{figure}

\begin{definition}\label{definition_balanced}
Suppose $G$ is a finite simple connected graph and $A$ is a set of vertices in $G$.
\begin{enumerate}
\item We say that $A$ is \emph{balanced in $G$} if for every equitable $2$-partition $\mathcal{D}$ of $V(G)$ with connected blocks, the cardinalities of $A$ intersected with the two blocks of $\D$ differ by at most one.
\item We say that $A$ is \emph{weakly balanced in $G$} if and only if for every equitable $2$-partition $\D$ of $V(G)$ with connected blocks, the cardinalities of $A$ intersected with the two blocks of $\D$ differ by at most two, and furthermore, whenever the cardinalities of $A$ intersected with the two blocks of such a $\D$ differ by exactly $2$, the block containing more points of $A$ must be the larger block, that is, $D_1,D_2\in\mathcal{D}$ with $|A\cap D_1|=|A\cap D_2|+2$ implies $|D_1|>|D_2|$.
\end{enumerate}
\end{definition}

\begin{theorem}[{Bushaw, Cody and Leffler \cite{BCL}}]\label{theorem_balanced}
Suppose $A$ is a set of vertices in $C_n$ where  $2\leq|A|=m\leq n$. 
\begin{enumerate}
\item\label{item_thm_bal_balanced} Suppose $n$ is even or $m$ is odd. Then $A$ is a maximizer of $W$ on $C_n$ if and only if $A$ is balanced.
\item\label{item_thm_bal_wbnb} Suppose $n$ is odd and $m$ is even. Then $A$ is a maximizer of $W$ on $C_n$ if and only if $A$ is weakly balanced. Furthermore, in this case, every maximizer of the Wiener index on $C_n$ with cardinality $m$ is not balanced.
\item\label{item_thm_bal_weakly_balanced} $A$ is a maximizer of $W$ if and only if it is weakly balanced.
\end{enumerate}
\end{theorem}

\begin{definition}\label{definition_good}
For an integer $n\geq 1$, we say that a set of vertices $A$ of a cycle $C_n$ is \emph{almost good} if either $|A|=1$, $|A|=n-1$ or the induced subgraph $C_n[A]$ has exactly two connected components whose vertex sets $A_1$ and $A_2$ satisfy $\left||A_1|-|A_2|\right|\leq 1$. We say that $A$ is \emph{good} if both $A$ and $V(C_n)\setminus A$ are almost good. 
\end{definition}

\begin{observation}\label{lemma_good_sets}
Suppose $m$ and $n$ are integers with $n\geq 4$, $2\leq m\leq n-2$. A set of vertices $A\subseteq V(C_n)$ is good if and only if there is an automorphism $\phi$ of $C_n$ such that 
\[A=\phi\left(\left[0,\floor{\frac{m}{2}}\right]\cup\left[\floor{\frac{n}{2}},\floor{\frac{n}{2}}+\ceil{\frac{m}{2}}-1\right]\right).\]
\end{observation}

\begin{lemma}\label{lemma_good_implies_maximier}
For an integer $n\geq 1$, if a set of vertices $A$ of $C_n$ is good then it is a maximizer of $W$.
\end{lemma}

\begin{proof}
Suppose $A$ is good and let $m=|A|$.
If $m=1$ there are no pairs of vertices in $A$, so $A$ is a maximizer of $W$. 
Similarly, if $m=n-1$, then $|V(C_n) \setminus A|=1$ and so $V(C_n)\setminus A$ is a maximizer of $W$, which implies that $A$ must be a maximizer of $W$ as well \cite{BCL}.

It remains to handle the case where $2\leq m\leq n-2$. Suppose $n$ is even. Then it is clear that whenever $\mathcal{D}=\{D_1,D_2\}$ is an equitable $2$-partition of $V(C_n)$ with connected blocks we have $|D_1|=|D_2|=\frac{n}{2}$ and, using Observation \ref{lemma_good_sets}, the quantities $|A\cap D_1|$ and $|A\cap D_2|$ differ by at most $1$. Hence, $A$ is balanced and by Theorem \ref{theorem_balanced}, $A$ is a maximizer of $W$ on $C_n$.

Suppose $n$ and $m$ are both odd. Let $\mathcal{D}=\{D_1,D_2\}$ be an equitable $2$-partition of $V(C_n)$ with connected blocks. Since $m$ is odd, it follows that $A$ is the union of two disjoint sets of vertices $A=A_1\cup A_2$ with $|A_1|=\frac{m-1}{2}$ and $|A_2|=\frac{m+1}{2}$, and where $C_n[A_1]$ and $C_n[A_2]$ are connected. If $D_1\cap A_2=\emptyset$ then $|D_1\cap A|=\frac{m-1}{2}$ and $|D_2\cap A|=\frac{m+1}{2}$. On the other hand, if $D_1\cap A_2\neq\emptyset$, then $|D_1\cap A|=\frac{m+1}{2}$ and $|D_2\cap A|=\frac{m-1}{2}$. In any case, $|D_1\cap A|$ and $|D_2\cap A|$ differ by at most one, so $A$ is balanced. Thus, by Theorem \ref{theorem_balanced}, $A$ is a maximizer of $W$ on $C_n$.

Suppose $n$ is odd and $m$ is even. Let $\mathcal{D}=\{D_1,D_2\}$ be an equitable $2$-partition of $V(C_n)$ with connected blocks such that $|D_1|=\frac{n+1}{2}$ and $|D_2|=\frac{n-1}{2}$. It follows from Observation \ref{lemma_good_sets} that 
\[\frac{m}{2}\leq |D_1\cap A|\leq\frac{m}{2}+1.\]
If $|D_1\cap A|=\frac{m}{2}$ then $|D_2\cap A|=\frac{m}{2}$. If $|D_1\cap A|=\frac{m}{2}+1$ then $|D_2\cap A|=\frac{m}{2}-1$. In either case $|D_1\cap A|$ and $|D_2\cap A|$ differ by at most two, and moreover, when the difference is two, it is the large block that contains more points of $A$. Therefore, $A$ is weakly balanced, and by Theorem \ref{theorem_balanced}, it is a maximizer of $W$ on $C_n$.
\end{proof}

\begin{lemma}\label{lemma_two_pieces}
    Suppose $A\subset V(C_n)$ is good and $|A|=\ell_1+\ell_2$ is an integer partition of $|A|$. Then there exists a partition $\mathcal{A}=\{A_1,A_2\}$ of $A$ such that $A_1$ and $A_2$ are good with $|A_1|=\ell_1$ and $|A_2|=\ell_2$.
\end{lemma}

\begin{proof}
Without loss of generality we can assume that $\ell_1,\ell_2 \geq 2$. Let $m=|A|$ and 
\[A=\{a_1,\ldots,a_{\floor{\frac{m}{2}}}\}\cup\{a_{\floor{\frac{m}{2}}+1},\ldots,a_m\}\]
where $\{a_1,\ldots,a_{\floor{\frac{m}{2}}}\}$ and $\{a_{\floor{\frac{m}{2}}+1},\ldots,a_m\}$ are the connected components of $A$.

Suppose $\ell_1$ and $\ell_2$ have the same parity. Then $m$ is even. Let 
\[A_1=\left\{a_1,\ldots,a_{\floor{\frac{\ell_1}{2}}}\right\}\cup \left\{ a_{\frac{m}{2}+1},\ldots,a_{\frac{m}{2}+\ceil{\frac{\ell_1}{2}}}\right\}\] and $A_2=A\setminus A_1$. It is easy to check that $\mathcal{A}=\left\{A_1,A_2\right\}$ is as desired.

Suppose $\ell_1$ and $\ell_2$ have different parity. Without loss of generality, let $\ell_1$ be even and $\ell_2$ be odd. Then $m$ is odd. Let
\[A_1=\left\{a_1,\ldots,a_{\frac{\ell_1}{2}}\right\}\cup\left\{a_{\floor{\frac{m}{2}}+1},\ldots,a_{\ceil{\frac{m}{2}}+\frac{\ell_1}{2}}\right\}\]
and $A_2=A\setminus A_1$. Then, again, one may easily check that $\mathcal{A}$ is as desired.
\end{proof}

\begin{lemma}\label{lemma_good_maximizers}
Suppose $n=\ell_1+\cdots+\ell_k$ is an integer partition of a positive integer $n$. Then there is a partition $\mathcal{A}=\{A_1,\ldots,A_k\}$ of $V(C_n)$ such that $A_i$ is good and $|A_i|=\ell_i$ for all $i\in \{1,\ldots,k\}$.
\end{lemma}


\begin{proof}
We will inductively construct $\mathcal{A}$. Using Lemma \ref{lemma_two_pieces}, let $\mathcal{A}_1=\{A_1,A_2'\}$ be a partition of $A_1'=V(C_n)$ such that $A_1$ and $A_2'$ are good, $|A_1|=\ell_1$ and $|A_2'|=\ell_2+\cdots+\ell_k$. 
Suppose $1\leq j < k$ and that for all $i$ with $1\leq i\leq j$ we've already defined a partition $\mathcal{A}_i=\{A_i,A_{i+1}'\}$ of $A_i'$ such that $A_i$ and $A_{i+1}'$ are good, $|A_i|=\ell_i$ and $|A_{i+1}'|=\ell_{i+1}+\cdots+\ell_k$. By Lemma \ref{lemma_two_pieces}, there is a partition $\mathcal{A}_{i+1}=\left\{A_{i+1},A_{i+2}'\right\}$ of $A_{i+1}'$ such that $A_{i+1}$ and $A_{i+2}'$ are good, $|A_{i+1}|=\ell_{i+1}$ and $|A_{i+2}'|= \ell_{i+2}+\cdots+\ell_k$. Let $\mathcal{A}=\{A_1,\ldots,A_k\}$ where $A_k=A_k'$. It follows that $\mathcal{A}$ is the desired partition of $V(C_n)$.
\end{proof}

\begin{theorem}\label{theorem_weak_max_on_C_n}
A surjective function $f:V(C_n)\to\{1,\ldots,k\}$ is a $k$-color weak maximizer of $W$ on $C_n$ if and only if for all $i\in\{1,\ldots,k\}$ the set $V_i=f^{-1}(i)$ is a maximizer of $W$.
\end{theorem}

\begin{proof}
The backwards direction is immediate. Let us prove the forward direction. For the forward direction, suppose $V_1=f^{-1}(1)$ is not a maximizer of $W$. We must show that $f$ is not a $k$-color weak maximizer of $W$. It suffices to show that there exists a surjective function $f':V(C_n)\to\{1,\ldots,k\}$ with the same type as $f$ such that all of the color classes of $f'$ are maximizers of $W$ on $C_n$. But, this easily follows from Lemma \ref{lemma_good_implies_maximier} and Lemma \ref{lemma_good_maximizers}.
\end{proof}

\section{Local weak maximizers on paths}\label{section_local_weak_max_paths}

Suppose $n$ and $k$ are positive integers and $k\leq n$. For a given tuple of positive integers $t=(n_1,\ldots,n_k)$ such that 
\[n_1\leq \cdots\leq n_k\]
and
\[n_1+\cdots+n_k=n,\] we define a collection $\C_t$ of surjective $k$-colorings of $V(P_n)$ that will coincide with the collection of $k$-color weak maximizers of $W$ on $P_n$ of type $t$, as well as with the set of $k$-color local weak maximizers of $W$ on $P_n$ of type $t$. 

Before we give the full definition of $\C_t$, which will be required below, let us first provide an equivalent definition,\footnote{We would like to thank an anonymous referee for suggesting this simplified definition of $\C_t$.} which may aid the reader in comprehending the longer and more notationally complex definition. We build up a coloring of $V(P_n)$ in $\C_t$ in stages, starting with all vertices uncolored. At stage $i$, each color $j\in\{1,\ldots,k\}$ is assigned a \emph{capacity} $r^i_j$, which is the number of vertices required to be that color, $n_i$, minus the number of vertices that have already been assiged color $j$. We repeat the following step until the maximum of all capacities $r^i_j$, over all colors, at current stage $i$, is less than $2$: let $K\geq 1$ be the number of colors whose capacities at the current stage are equal to the maximum of all capacities at the same stage, assign one of the least $K$ uncolored vertices to each of these colors, and assign one of the greatest $K$ uncolored vertices to each of these colors (in any order), thus decreasing the capacities of these color by either $1$ or $2$. Finally, when all capacities are $0$ or $1$, assign one of the remaining vertices to each color with capacity $1$ (if any). Then $\C_t$ is the class of all colorings $f:V(P_n)\to\{1,\ldots,k\}$ that can be obtained in this way.

Now let us give the full definition of $\C_t$: If $x=(x_1,\ldots,x_k)$ is a finite tuple of real numbers, we let 
\[M(x)=\{i\st 1\leq i\leq k\text{ and }x_i=\max_{1\leq j\leq k}x_j\}.\]
We define a sequence of $k$-tuples $r^1,r^2,\ldots$, where $r^i=(r^i_1,\ldots,r^i_k)$ as follows. 
\begin{itemize}
    \item Let $r^1=t=(n_1,\ldots,n_k)$. 
    \item Given $r^i$, if $\max(r^i)>1$ then we let $r^{i+1}$ be the $k$-tuple obtained from $r^i$ by subtracting $2$ from each max in $r^i$; that is, 
\[r^{i+1}_j =\begin{cases}
    r^i_j-2 & \text{ if $j\in M(r^i)$}\\
    r^i_j & \text{ otherwise}.\\
\end{cases}\]
    \item Given $r^i$, if $\max(r^i)=1$ then we let $r^{i+1}$ be the $k$-tuple obtained from $r^i$ by subtracting $1$ from each max in $r^i$; that is, 
\[r^{i+1}_j =\begin{cases}
    r^i_j-1 & \text{ if $j\in M(r^i)$}\\
    r^i_j & \text{otherwise}.\\
    \end{cases}
    \]
\end{itemize}
Let $i^*$ be the greatest positive integer $i$ for which $\max(r^i)>0$. For $1\leq i\leq i^*$, let $m_i=|M(r^i)|$ be the number of maximums that occur in $r^i$. 

We use the $m_i$'s and $r^i$'s, which are computed using the type $t$, for $1\leq i\leq i^*$ to define a family $\C_t$ of surjective functions $f:V(P_n)\to\{1,\ldots,k\}$. 
We define a partition $\mathcal{D}_t$ of $V(P_n)$ as follows. Let
\begin{align*}
D_1^L&=[1,m_1]\text{\ \ \ and}\\
D_1^R&=[n-m_1+1,n].
\end{align*}
For $1<i\leq i^*$, let
\begin{align*}
D_i^L&=[1+m_1+\cdots+m_{i-1}, m_1+\cdots+m_i]\text{\ \ \ and}\\
D_i^R&=[n-(m_1+\cdots+m_i)+1, n-(m_1+\cdots+m_{i-1})].
\end{align*}
If $\max(r^{i^*})>1$, we let \[\mathcal{D}=\{D_1^L,D_2^L,\ldots,D_{i^*-1}^L,D_{i^*}^L,D_{i^*}^R,D_{i^*-1}^R,\ldots,D_2^R,D_1^R\},\]
and if $\max(r^{i^*})=1$, we let
\[\mathcal{D}=\{D_1^L,D_2^L,\ldots,D_{i^*-1}^L,D_{i^*}^L,D_{i^*-1}^R,\ldots,D_2^R,D_1^R\}.\]

We define $\C_t$ to be the family of functions $f:V(P_n)\to \{1,\ldots,k\}$ such that:
\begin{enumerate}
    \item if $\max(r^{i^*})>1$, then for $1\leq i\leq i^*$, $f\restrict D_i^L:D_i^L\to M(r^i)$ and $f\restrict D_i^R:D_i^R\to M(r^i)$ are both bijections, and
    \item if $\max(r^{i^*})=1$, then for $1\leq i<i^*$, $f\restrict D_i^L:D_i^L\to M(r^i)$ and $f\restrict D_i^R:D_i^R\to M(r^i)$ are both bijections, and  $f\restrict D_{i^*}^L:D_{i^*}^L\to M(r^{i^*})$ is a bijection.
\end{enumerate}

It follows easily, by construction, that the functions in $\C_t$ are surjective to $\{1,\ldots,k\}$.

\begin{lemma}\label{lemma_swapping_surj}
If $f\in \C_t$ then $f:V(P_n)\to\{1,\ldots,k\}$ is a surjective $k$-coloring. 
\end{lemma}

In the following examples we illustrate the construction and visualization of the collection $\C_t$ for a particular value of $t$.

\begin{example}
Let us consider the collection $\C_t$, where $t=(2,6,6)$ and hence $n=14$. To construct $\C_{(2,6,6)}$ we first compute the following using the above definitions:
\[\begin{tabular}{l@{\hspace{1cm}} l@{\hspace{1cm}} l@{\hspace{1cm}} l}
$r^1=(2,6,6)$ & $m_1=2$ & $D_1^L=[1,2]$ & $D_1^R=[13,14]$\\
$r^2=(2,4,4)$ & $m_2=2$ & $D_2^L=[3,4]$ & $D_2^R=[11,12]$\\
$r^3=(2,2,2)$ & $m_3=3$ & $D_3^L=[5,7]$ & $D_3^R=[8,10]$\\
\end{tabular}
\]
The collection $\C_t$ is represented in Figure \ref{figure_P_14_a} in the sense that not only is the displayed coloring $f$ in $\C_{(2,6,6)}$, but so too is any coloring obtained from $f$ by permuting colors within individual blocks of the partition $\D_{(2,6,6)}=\{D_1^L,D_2^L,D_3^L,D_3^R,D_2^R,D_1^R\}$. 
Recall that $S_{6,7}(f)$ is the coloring obtained from $f$ by swapping the colors of vertex $6$ and $7$. Notice that $W(S_{6,7}(f))-W(f)=0$. Indeed, for any member of $\C_{(2,6,6)}$, permuting colors of vertices within blocks of the partition $\D_{(2,6,6)}$ does not change the Wiener index; this holds in general by Corollary \ref{corollary_C_t_same_Wiener_index}.
\end{example}

\begin{figure}

\centering
\begin{subfigure}[b]{\textwidth}

\centering
\begin{tikzpicture}[scale=0.7]

\definecolor{color3}{rgb}{0.9,0.9,0.9}    
\definecolor{color2}{rgb}{0.5,0.5,0.5}    
\definecolor{color1}{rgb}{0.2,0.2,0.2}    

\definecolor{color4}{rgb}{1, 0.4, 0.4}    
\definecolor{color5}{rgb}{0.4, 0.4, 1}    
\definecolor{color6}{rgb}{1, 1, 0}    

\draw (1,0)--(14,0);

\foreach \i in {1,3,5,8,12,14}{
    \node[draw, fill=color4, circle, minimum size=2mm, inner sep=0pt] (\i) at (\i, 0) {};
}

\foreach \i in {2,4,6,9,11,13}{
    \node[draw, fill=color6, circle, minimum size=2mm, inner sep=0pt] (\i) at (\i, 0) {};
}

\foreach \i in {7,10}{
    \node[draw, fill=color5, circle, minimum size=2mm, inner sep=0pt] (\i) at (\i, 0) {};
}

    \foreach \i in {1, ..., 14} {
        \node at (\i, -0.8) {\i};  
    }

\foreach \i/\j in {1/2,3/2,5/3}{

\draw[rounded corners=2mm] (\i-0.4, -0.4) rectangle (\i+\j-1+0.4, 0.4) {};

\draw[rounded corners=2mm] (14-\i+1+0.4, -0.4) rectangle (14-\i+1-\j+1-0.4, 0.4) {};

}

\foreach \i/\x in {1/1.5,2/3.5,3/6}{
    \node at (\x, 1) {$D_\i^L$};  
}
\foreach \i/\x in {1/13.5,2/11.5/5,3/9}{
    \node at (\x, 1) {$D_\i^R$};  
}
    
\end{tikzpicture}
\caption{\small The collection $\C_{(2,6,6)}$.}
\label{figure_P_14_a}
\end{subfigure}

\begin{subfigure}[b]{\textwidth}

\centering

\begin{tikzpicture}[scale=0.7]

\definecolor{color3}{rgb}{0.9,0.9,0.9}    
\definecolor{color2}{rgb}{0.5,0.5,0.5}    
\definecolor{color1}{rgb}{0.2,0.2,0.2}    

\definecolor{color4}{rgb}{1, 0.4, 0.4}    
\definecolor{color5}{rgb}{0.4, 0.4, 1}    
\definecolor{color6}{rgb}{1, 1, 0}    

\draw (1,0)--(14,0);

\foreach \i in {1,3,6,9,12,14}{
    \node[draw, fill=color4, circle, minimum size=2mm, inner sep=0pt] (\i) at (\i, 0) {};
}

\foreach \i in {2,4,7,10,13}{
    \node[draw, fill=color6, circle, minimum size=2mm, inner sep=0pt] (\i) at (\i, 0) {};
}

\foreach \i in {5,8,11}{
    \node[draw, fill=color5, circle, minimum size=2mm, inner sep=0pt] (\i) at (\i, 0) {};
}

    \foreach \i in {1, ..., 14} {
        \node at (\i, -0.8) {\i};  
    }

\foreach \i/\j in {1/1,2/1,3/1,4/2,6/1,7/2}{

\draw[rounded corners=2mm] (\i-0.4, -0.4) rectangle (\i+\j-1+0.4, 0.4) {};

\draw[rounded corners=2mm] (14-\i+1+0.4, -0.4) rectangle (14-\i+1-\j+1-0.4, 0.4) {};

}

\foreach \i/\x in {1/1,2/2,3/3,4/4.5,5/6,6/7.5}{
    \node at (\x, 1) {$D_\i^L$};  
}
\foreach \i/\x in {1/14,2/13,3/12,4/10.5,5/9}{
    \node at (\x, 1) {$D_\i^R$};  
}
    
\end{tikzpicture}
\caption{\small The collection $\C_{(3,5,6)}$.}
\label{figure_P_14_b}
\end{subfigure}

\caption{}
\label{figure_P_14}
\end{figure}

\begin{example}
Let us consider the collection $\C_t$ where $t=(3,5,6)$. Again, using the definitions above we can easily compute the following. 
\[
\begin{tabular}{l@{\hspace{1cm}} l@{\hspace{1cm}} l@{\hspace{1cm}} l}
$r^1 = (3, 5, 6)$ & $m_1 = 1$ & $D_1^L = [1, 1]$ & $D_1^R = [14, 14]$ \\
$r^2 = (3, 5, 4)$ & $m_2 = 1$ & $D_2^L = [2, 2]$ & $D_2^R = [13, 13]$ \\
$r^3 = (3, 3, 4)$ & $m_3 = 1$ & $D_3^L = [3, 3]$ & $D_3^R = [13, 13]$ \\
$r^4 = (3, 3, 2)$ & $m_4 = 2$ & $D_4^L = [4, 5]$ & $D_4^R = [10, 11]$ \\
$r^5 = (1, 1, 2)$ & $m_5 = 1$ & $D_5^L = [6, 6]$ & $D_5^R = [9, 9]$ \\
$r^6 = (1, 1, 0)$ & $m_6 = 2$ & $D_6^L = [7, 8]$ & \\
\end{tabular}
\]

As in the previous example, the collection $\C_{(3,5,6)}$ is represented in Figure \ref{figure_P_14_b}. Suppose that we want to swap the colors of vertex 4 and vertex 5. We have that  $L_4^f=0$, $R_4^f=2$, $L_5^f=1$, and $R_5^f=3$. 
Hence, $L_4^f-R_4^f=-2$
and $R_5^f-L_5^f=2$, with $\max(r^4)=3$. We see that, $W(S_{4,5}(f))-W(f)=0$.
\end{example}

Let us now verify that the collection $\C_t$ is closed under the operations of swapping colors of adjacent vertices within blocks of the partition $\D_t$. In order to do this, we will compute the change in Wiener index after a color swap in terms of several quantities that we define now. Suppose $f:V(P_n)\to\{1,\ldots,k\}$. For $v\in V(P_n)$ we let $L_v^f$ denote the number of vertices of $P_n$ to the left of $v$ that are assigned color $f(v)$; that is,
\[L_v^f=\left|[1,v-1]\cap f^{-1}(f(v))\right|.\]
Similarly, $R_v^f$ is the number of vertices of $P_n$ to the right of $v$ that are assigned color $f(v)$ by $f$; that is,
\[R_v^f=\left|[v+1,n]\cap f^{-1}(f(v))\right|.\]
Notice that after swapping the color of vertex $v$ and $v+1$, the change in Wiener index is
\[W(S_{v,v+1}(f))-W(f)=L_v^f-R_v^f+R_{v+1}^f-L_{v+1}^f.\]

\begin{lemma}\label{lemma_swapping_colors}
Suppose $f\in \C_t$ and $D\in\D_t$ such that $v,v+1\in D$. Then the coloring $S_{v,v+1}(f)$ obtained from $f$ by swapping the colors of vertex $v$ and $v+1$ has the same Wiener index as $f$. That is, $W(S_{v,v+1}(f))=W(f)$.
\end{lemma}

\begin{proof}
Without loss of generality, let us assume that $v,v+1\in D_i^L$ where $1\leq i\leq i^*$. We have 
\begin{align*}
W(S_{v,v+1}(f))-W(f)&=L_v^f-R_v^f+R_{v+1}^f-L_{v+1}^f.\\
\end{align*}
But, by construction 
\[L_v^f-R_v^f=-(\max(r^i)-1)\]
and 
\[R_{v+1}^f-L_{v+1}^f=\max(r^i)-1,\]
hence $W(S_{v,v+1}(f))-W(f)=0$.
\end{proof}

Since any function in $\C_t$ can be obtained from any other by applying a finite sequence of color swaps between adjacent vertices in blocks of the partition $\D_t$, we easily obtain the following.

\begin{corollary}\label{corollary_C_t_same_Wiener_index}
If $f,g\in \C_t$ then $W(f)=W(g)$.
\end{corollary}

We are ready to prove the main result of this section.

\begin{theorem}
Suppose $g:V(P_n)\to\{1,\ldots,k\}$ is a surjective $k$-coloring and $t=\type(g)=(n_1,\ldots,n_k)$. The following are equivalent.
\begin{enumerate}
\item $g$ is a $k$-color weak maximizer of $W$ on $P_n$.
\item $g$ is a $k$-color local weak maximizer of $W$ on $P_n$.
\item $g\in \C_t$.
\end{enumerate}
\end{theorem}

\begin{proof}
(1) implies (2) follows easily from the definitions. Let us prove (2) implies (3). Suppose $g:V(P_n)\to \{1,\ldots,k\}$ is surjective with $\type(g)=t$ and $g\notin\C_t$. We must show that $g$ is not a local weak maximizer on $P_n$. We will use the definitions given above for various objects defined in terms of the type $t$, which are independent of $g$.

Suppose there is an $i$ with $1\leq i\leq i^*$ such that $|\max(r^i)|>1$ but either $g\restrict D_i^L$ or $g\restrict D_i^R$ is not a bijection to $M(r^i)$. Let $i_0$ be the least such $i$ and assume, without loss of generality, that $g\restrict D_{i_0}^L$ is not a bijection from $D_{i_0}^L$ to $M(r^{i_0})$. There are two possibilities. Case 1: Suppose there is some $v$ in $D_{i_0}^L$ with $g(v)\notin M(r^{i_0})$, and let $v_0$ be the least such $v$. Let $v_1$ be the least $v$ greater than $v_0$ such that $g(v)\in M(r^{i_0})$. Then
\[W(S_{v_1-1,v_1}(g))-W(g)=R_{v_1}^g-L_{v_1}^g+L_{v_1-1}^g-R_{v_1-1}^g.\]
Since there is no $v\in[v_0,v_1-1]$ with $g(v)\in M(r^{i_0})$, it follows that 
\[R_{v_1}^g-L_{v_1}^g>R_{v_1-1}^g-L_{v_1-1}^g,\]
and thus 
\[W(S_{v_1-1,v_1}(g))>W(g).\]
Case 2: Suppose $g(D_{i_0}^L)\subseteq M(r^{i_0})$, so $g\restrict D_{i_0}^L:D_{i_0}^L\to M(r^{i_0})$, but $g\restrict D_{i_0}^L$ is not bijective. Let $c\in M(r^{i_0})$ be a color such that there are at least two vertices $v_0,v_1\in D_{i_0}^L$ with $v_0<v_1$ and $g(v_0)=g(v_1)=c$. Let $v_2$ be the least vertex $v$ greater than $v_1$ such that $g(v)\in M(r^{i_0})\setminus\{c\}$. Then, it follows that 
\[R_{v_2}^g-L_{v_2}^g>R_{v_2-1}^g-L_{v_2-1}^g\]
and thus
\[W(S_{v_2-1,v_2}(g))>W(g).\]


On the other hand, suppose that whenever $1\leq i\leq i^*$ and $|\max(r^i)|>1$, it follows that $g\restrict D_i^L:D_i^L\to M(r^i)$ and $g\restrict D_i^R:D_i^R\to M(r^i)$ are both bijections, but $\max(r^{i^*})=1$ and $g\restrict D_{i^*}^L$ is not a bijection to $M(r^{i^*})$. This is not possible because we are assuming $\type(g)=t$ and $|D_{i^*}^L|=|M(r^{i^*})|=m_{i^*}$.

Now let us show that (3) implies (1). Suppose $g\in \C_t$. We must show that $g$ is a weak maximizer of $W$ on $P_n$. Let $f$ be a weak maximizer of $W$ on $P_n$ with $\type(f)=\type(g)=t$. Since we already showed that every weak maximizer of $W$ on $P_n$ is in $\C_t$ (see (1) implies (3)), it follows that $f\in \C_t$. Since $f,g\in\C_t$, it follows by Corollary \ref{corollary_C_t_same_Wiener_index} that $W(g)=W(f)$ and hence $g$ is a weak maximizer of $W$ on $P_n$.
\end{proof}

\section{Majorization on types of colorings and the Wiener index}\label{section_majorization}


\subsection{Majorization}

Let us introduce the majorization order $\maj$ on $\R^k$. For $x=(x_1,\ldots,x_k)\in\R^k$, we let $x_{(1)}\leq \cdots\leq x_{(k)}$ denote the components of $x$ in non-decreasing order and let $x^{\uparrow}=(x_{(1)},\ldots,x_{(k)}).$
Similarly, let $x_{[1]}\geq\cdots\geq x_{[k]}$
denote the components of $x$ in non-increasing order and define $x^{\downarrow}=(x_{[1]},\ldots,x_{[k]})$.

\begin{definition}\label{definition_maj} Suppose $x,y\in\R^n$. We define 
\[x\maj y \ \ \text{if}\ \ \begin{cases}\displaystyle\sum_{j=1}^sx_{[j]}\leq\sum_{j=1}^sy_{[j]}$ for $s=1,\ldots,k-1,\text{ and}\vspace{3mm}\\ \displaystyle\sum_{j=1}^kx_{[j]}=\sum_{j=1}^ky_{[j]}.\end{cases}\] 
When $x\maj y$ we say that $x$ is \emph{majorized} by $y$, or that $y$ \emph{majorizes} $x$.
\end{definition}

Intuitively speaking, $x\maj y$ indicates that $x$ is ``more equitable'' than $y$. For example, \[\left(\frac{10}{4},\frac{10}{4},\frac{10}{4},\frac{10}{4}\right)\maj(2,2,3,3)\maj(1,1,3,5)\maj(1,1,1,7).\] For background and more details on majorization, see \cite{MR2759813}.

In what follows, for notational simplicity, we will sometimes restrict our attention to the set of nondecreasing $k$-tuples of integers, which we will denote by 
\[\Z^k_\uparrow=\{x^\uparrow\st x\in\Z^k\}.\]
Given integers $1\leq j\leq j'\leq k$ and a tuple $x=(x_1,\ldots,x_k)\in\Z^k_\uparrow$ we define $R_{j,j'}(x)$ to be the tuple obtained from $x$ by subracting one from the $j$th entry of $x$, adding one to the $j'$th entry of $x$ and then by reindexing if necessary so that the resulting tuple is arranged in nondecreasing order. That is,
\[R_{j,j'}(x_1,\ldots,x_k)=(x_1,\ldots,x_{j-1},x_j-1,x_{j+1},\ldots,x_{j'-1},x_{j'}+1,x_{j'+1},\ldots,x_k)^\uparrow.\]
When $j<j'$ we refer to the function $R_{j,j'}:\Z^k_\uparrow\to\Z^k_\uparrow$ as the \emph{reverse Robin Hood transfer from $j$ to $j'$}. The following lemma is due to Murihead \cite{Murihead1903} (see \cite[Lemma 5.D.1]{MR2759813} or  \cite[Lemma 4.1]{MR262112} for a proof).

\begin{lemma}[Murihead, 1903]\label{lemma_reverse_robin_hood}
For two finite tuples of integers $x=(x_1,\ldots,x_k)$ and $y=(y_1,\ldots,y_k)$, written in nondecreasing order, we have $x\maj y$ if and only if there is a finite sequence of reverse Robin Hood transfers $R_{j_1,j_1'},\ldots,R_{j_\ell,j_\ell'}$ such that
\[R_{j_\ell,j_\ell'}\circ\cdots \circ R_{j_1,j_1'}(x) = y.\]
\end{lemma}

In what follows we will need the following folklore lemma (see \cite[Lemma 2.15]{BCL}).

\begin{lemma}\label{lemma_equitable_tuples}
Suppose $n$ and $k$ are integers with $k\geq 1$. 

\begin{enumerate}
    \item There exists a unique $k$-tuple of positive integers $t(n,k)=(m_1,\ldots,m_k)$ such that
    \begin{itemize}
        \item $\sum_{j=1}^km_j=n$,
        \item $m_1\leq\cdots\leq m_k$ and
        \item $|m_j-m_{j'}|\leq 1$ for all $j,j'\in\{1,\ldots,k\}$. 
    \end{itemize}
    Furthermore, for such a $k$-tuple we have $m_j\in\left\{\floor{\frac{n}{k}},\ceil{\frac{n}{k}}\right\}$ for all $j\in\{1,\ldots,k\}$.
    \item Suppose $n$ and $k$ are integers with $k\geq 1$. If $t=(n_1,\ldots,n_k)$ is a tuple of integers with $\sum_{j=1}^kn_j= n$, then $t(n,k)\maj t$.
\end{enumerate}
\end{lemma}

\subsection{Majorization and colorings of paths}

We will determine which $k$-color weak maximizers of $W$ on $P_n$ are maximizers of $W$ and which have the smallest possible Wiener index. To do this, we establish a connection between the majorization order on types of colorings and the Wiener index of $k$-color weak maximizers of $W$ on paths.

\begin{theorem}\label{theorem_maj_paths}
Suppose $f,f':V(P_n)\to\{1,\ldots,k\}$ are $k$-color weak maximizers of $W$ on $P_n$. If $\type(f)\maj \type(f')$ and $\type(f)\neq\type(f')$ then $W(f)<W(f')$.
\end{theorem}

\begin{proof}
Suppose $f:V(P_n)\to\{1,\ldots,k\}$ is a $k$-color weak maximizer of $W$ with $\type(f)=(n_1,\ldots,n_k)$. Suppose $1\leq j<j'\leq k$, so that $n_j\leq n_{j'}$. By Lemma \ref{lemma_reverse_robin_hood}, it will suffice to show that if $g:V(P_n)\to\{1,\ldots,k\}$ is a $k$-color weak maximizer of $W$ with $\type(g)=R_{j,j'}(\type(f))$, that is
\[\type(g)=(n_1,\ldots,n_{j-1},n_j-1,n_{j+1},\ldots, n_{j'-1},n_{j'}+1,n_{j'+1},\ldots,n_k)^\uparrow,\]
then $W(f)<W(g)$.

Suppose $g$ is such a $k$-color weak maximizer of $W$ on $P_n$. We would like to show that it is possible to modify $f$ by changing the color of one vertex in $f^{-1}(j)$ to be color $j'$ in such a way that the Wiener index of the newly obtained coloring $\tilde{f}$ is strictly greater than $W(f)$. If we were able to do this, it would then follow that $\type(\tilde{f})=\type(g)$ and hence
\[W(g)\geq W(\tilde{f})>W(f)\]
as desired. As it turns out, this is not always possible. However, if we first prepare for such a color change by identifying another $k$-color weak maximizer $f_*$ of $W$ on $P_n$ with $\type(f_*)=\type(f)$ and $W(f_*)=W(f)$, then it will always be possible to modify $f_*$ by changing the color of some vertex in $f_*^{-1}(j)$ to be color $j'$ to obtain a new coloring $f_{**}$ with $\type(f_{**})=\type(g)$ and $W(g)\geq W(f_{**})>W(f_*)=W(f)$.

Therefore, it only remains to prove that there is a $k$-color weak maximizer $f_*$ on $P_n$ with $\type(f_*)=\type(f)$ and $W(f_*)=W(f)$ for which there is a vertex $v\in f_*^{-1}(j)$ such that by changing the color of $v$ to be color $j'$ we obtain a new coloring $f_{**}$ with $W(f_{**})>W(f_*)$.

Recall that the functions in $\C_{\type(f)}$ are precisely the functions obtained from $f$ by permuting colors within the blocks of the partition $\D_{\type(f)}$ (see Section \ref{section_local_weak_max_paths} for definitions). We choose $f_*\in\C_{\type(f)}$ so that vertices of color $j$ are as close as possible to the center of $P_n$ and vertices of color $j'$ are as far as possible from the center of $P_n$. Additionally, if $D_i^L$ is a block of the partition $\D_{\type(f)}$ with $\frac{n}{2}\in D_i^L$, and if $j$ is an ``available color'' for a vertex in $D_i^L$, that is if $j\in M(r^i)$, then we require that $f_*\left(\frac{n}{2}\right)=j$. We let $f_{**}$ be the coloring obtained from $f_*$ by finding a center-most vertex of color $j$, call it $\hat{v}$, and changing it to have color $j'$. Then,
\[\type(f_{**})=(n_1,\ldots,n_{j-1},n_j-1,n_{j+1},\ldots, n_{j'-1},n_{j'}+1,n_{j'+1},\ldots,n_k)\]
and
\[W(f_{**})-W(f_*)\hspace{5pt} =\hspace{10pt} \sum_{\mathclap{v\in f_*^{-1}(j')}}d(\hat{v},v)\hspace{10pt} - \hspace{20pt}\sum_{\mathclap{v\in f_*^{-1}(j)\setminus\{\hat{v}\}}}d(\hat{v},v).\]
In order to obtain $W(f_{**})-W(f_*)>0$, it suffices to show that there is an injective function 
\[\varphi:f_*^{-1}(j)\setminus\{\hat{v}\}\to f_*^{-1}(j')\]
such that for each $v\in f_*^{-1}(j)\setminus\{\hat{v}\}$ we have
\begin{align}
    d(\hat{v},v)<d(\hat{v},\varphi(v)).\label{equation_distance_inequality}
\end{align}
For $v\in f_*^{-1}(j)\setminus\{\hat{v}\}$, it follows that $v\neq\frac{n}{2}$, and furthermore, since $n_j\leq n_{j'}$ the sets $f_*^{-1}(j')\cap[1,v-1]$ and $f_*^{-1}(j')\cap [v+1,n]$ are both nonempty. Thus we may define
\[\varphi(v)=\begin{cases}
    \max(f_*^{-1}(j')\cap[1,v-1]) & \text{ if $v<\frac{n}{2}$}\\
    \min(f_*^{-1}(j')\cap [v+1,n]) & \text{ if $v>\frac{n}{2}$,}
\end{cases}\]
for $v\in f_*^{-1}(j)\setminus\{\hat{v}\}$. It follows easily from the definitions of $f_*$ and $\varphi$ that (\ref{equation_distance_inequality}) holds for all $v\in f_*^{-1}(j)\setminus\{\hat{v}\}$. Let us show that $\varphi$ is injective. Suppose $v,v'\in f_*^{-1}(j)\setminus\{\hat{v}\}$ are such that $\varphi(v)=\varphi(v')$ and without loss of generality that $v\leq v'<\frac{n}{2}$. For contradiction assume that $v\neq v'$, so $v<v'$. Then $\max(f_*^{-1}(j')\cap[1,v-1])=\max(f_*^{-1}(j')\cap[1,v'-1])$. It follows that there is no vertex $u\in V(P_n)$ such that $v\leq u<v'$ and $f_*(u)=j'$. But this is not possible by the definition of $f_*$ and since $n_j\leq n_{j'}$.
\end{proof}

\begin{corollary}\label{corollary_largest_and_smallest_paths}
Suppose $2\leq k \leq n$ are integers and $f:V(P_n)\to \{1,\ldots,k\}$ is a $k$-color weak maximizer of $W$ on $P_n$. Then:
\begin{enumerate}
    \item $W(f)=\min\{W(f')\st f'\in \WM_k(P_n)\}$ if and only if the number of vertices in any two color classes of $f$ differ by at most one and
    \item $W(f)=\max\{W(f')\st f'\in \WM_k(P_n)\}$ if and only if $f$ has a color class of the largest possible size $n-(k-1)$.
\end{enumerate}
\end{corollary}

\begin{proof}
Let $f:V(P_n)\to\{1,\ldots,k\}$ be a $k$-color weak maximizer of $W$ on $P_n$. For the forward direction of (1), suppose that there exist two color classes of $f$ whose cardinalities differ by at least two. Let $t=(n_1,\ldots,n_k)=\type(f)$ and let $t(n,k)$ be the $k$-tuple defined in Lemma \ref{lemma_equitable_tuples}. It follows from Lemma \ref{lemma_equitable_tuples} that $t(n,k)\maj t$ and $t(n,k)\neq t$, and thus, by Theorem \ref{theorem_maj_paths} if $g$ is a $k$-color weak maximizer of $W$ on $P_n$ with $\type(g)=t(n,k)$, we see that $W(g)<W(f)$, and hence $W(f)\neq\min\{W(f')\st f'\in \WM_k(P_n)\}$. For the backward direction of (1), suppose the number of vertices in any two color classes of $f$ differ by at most one. Suppose $f'$ is any $k$-color weak maximizer of $W$ on $P_n$. Then by Lemma \ref{lemma_equitable_tuples} we have $\type(f)=t(n,k)\maj \type(f')$ and hence $W(f)\leq W(f')$.

For the forward direction of (2), suppose $f$ does not have a color class of the largest possible size $n-(k-1)$. Let $t$ be the $k$-tuple $t=(1,\ldots,1,n-(k-1))$ and choose some $g\in \C_t$. It follows that $\type(f)\neq t$ and $\type(f)\maj t$. Thus, by Theorem \ref{theorem_maj_paths}, $W(f)<W(g)$ and hence $W(f)\neq\max\{W(f')\st f'\in \WM_k(P_n)\}$. For the backward direction of (2), if $f$ has a color class of the largest possible size $n-(k-1)$, then it follows that $\type(f')\maj\type(f)$ for all $f'$ that are $k$-color weak maximizers of $W$ on $P_n$. Hence, by Theorem \ref{theorem_maj_paths}, $W(f)=\max\{W(f')\st f'\in \WM_k(P_n)\}$.
\end{proof}

\subsection{Majorization and colorings of cycles}

We would like to determine which $k$-color weak maximizers of $W$ on $C_n$ are maximizers of $W$ and which will have the smallest possible Wiener index. To do this, we establish a connection between the majorization order on types of colorings and the Wiener index of $k$-color weak maximizers on cycles (see Theorem \ref{theorem_maj_cycles} and Corollary \ref{corollary_W_does_not_go_down}). However, we find that the situation with cycles is more difficult than that of paths. As we did for paths in Theorem \ref{theorem_maj_paths}, we would like to show that if $f,f':V(C_n)\to\{1,\ldots,k\}$ are $k$-color weak maximizers with $\type(f)\maj\type(f')$ and $\type(f)\neq\type(f')$ then $W(f)<W(f')$. However, as the following example illustrates, this is not possible in general.

\begin{example}\label{example_bad}
Consider the $3$-color weak maximizers $f$ and $f'$ on $C_6$, shown in Figure \ref{figure_bad_example}. We have 
\[\type(f)=(2,2,2)\maj(1,2,3)=\type(f')\]
and of course $(2,2,2)\neq (1,2,3)$, but $W(f) = W(f')=9$. Thus, in general, for $k$-color weak maximizers $g,g':V(C_n)\to\{1,\ldots,k\}$, $\type(g)\maj\type(g')$ and $\type(g)\neq\type(g')$ does not always imply $W(g)<W(g')$. See Figure \ref{figure_3cwm_maj} for additional examples.

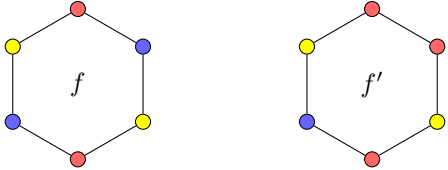
\begin{figure}[ht]
    \centering
    \begin{minipage}{0.3\textwidth}
        \centering
        \begin{tikzpicture}
        \definecolor{color4}{rgb}{1, 0.4, 0.4}    
        \definecolor{color5}{rgb}{0.4, 0.4, 1}    
        \definecolor{color6}{rgb}{1, 1, 0}    
            \foreach \i/\label in {1,4} {
                \node[draw, fill=color4, circle, minimum size=2mm, inner sep=0pt] (\i) at ({360/6 * (8-\i)+360/12}:1) {};
            }
            \foreach \i/\label in {2,5} {
                \node[draw, fill=color5, circle, minimum size=2mm, inner sep=0pt] (\i) at ({360/6 * (8-\i)+360/12}:1) {};
            }
            \foreach \i/\label in {3,6} {
                \node[draw, fill=color6, circle, minimum size=2mm, inner sep=0pt] (\i) at ({360/6 * (8-\i)+360/12}:1) {};
            }
            \node[] at (0,0) {$f$};
            \foreach \i/\j in {1/2, 2/3, 3/4, 4/5, 5/6, 6/1} {
                \draw (\i) -- (\j);
            }
        \end{tikzpicture}
    \end{minipage}
    \begin{minipage}{0.3\textwidth}
        \centering
        \begin{tikzpicture}
        \definecolor{color4}{rgb}{1, 0.4, 0.4}    
        \definecolor{color5}{rgb}{0.4, 0.4, 1}    
        \definecolor{color6}{rgb}{1, 1, 0}    
            \foreach \i/\label in {1,2,4} {
                \node[draw, fill=color4, circle, minimum size=2mm, inner sep=0pt] (\i) at ({360/6 * (8-\i)+360/12}:1) {};
            }
            \foreach \i/\label in {5} {
                \node[draw, fill=color5, circle, minimum size=2mm, inner sep=0pt] (\i) at ({360/6 * (8-\i)+360/12}:1) {};
            }
            \foreach \i/\label in {3,6} {
                \node[draw, fill=color6, circle, minimum size=2mm, inner sep=0pt] (\i) at ({360/6 * (8-\i)+360/12}:1) {};
            }
            \node[] at (0,0) {$f'$};
            \foreach \i/\j in {1/2, 2/3, 3/4, 4/5, 5/6, 6/1} {
                \draw (\i) -- (\j);
            }
        \end{tikzpicture}
        
    \end{minipage}
    \caption{Two $3$-color weak maximizers on $C_6$.}
    \label{figure_bad_example}
\end{figure}
\end{example}

Let us now show that, even though a direct analogue of Theorem \ref{theorem_maj_paths} for cycles is not true in general (by Example \ref{example_bad}), we can prove a similar result in a modified form. What's more, in Theorem \ref{theorem_maj_cycles}, Corollary \ref{corollary_Weiner_index_R_infty} and Corollary \ref{corollary_W_does_not_go_down}, we identify all instances of $k$-color weak maximizers $f,f':V(C_n)\to\{1,\ldots,k\}$ with $\type(f)\maj\type(f')$ and $\type(f)\neq\type(f')$ for which $W(f)\not<W(f')$ (see Figure \ref{figure_3cwm_maj} for an example).

\begin{theorem}\label{theorem_maj_cycles}
Suppose $f,f':V(C_n)\to\{1,\ldots,k\}$ are $k$-color weak maximizers of $W$ on $C_n$ and suppose that $j,j'\in\{1,\ldots,k\}$ are such that $\type(f')=R_{j,j'}(\type(f))$.
\begin{enumerate}
    \item\label{item_weird_case} If $n$ is even and $|f^{-1}(j)|=|f^{-1}(j')|=2a$ where $a$ is a positive integer, then $W(f)=W(f')$. 
    \item\label{item_main_case} Otherwise, $W(f)<W(f')$.
\end{enumerate}
\end{theorem}

\begin{proof}
For notational simplicity, let $d=\diam(C_n)=\floor{\frac{n}{2}}$ and let $\type(f)=(n_1,\ldots,n_k)$ throughout the proof. Note that, in particular we have $n_j=|f^{-1}(j)|$ and $n_{j'}=|f^{-1}(j')|$. For (\ref{item_weird_case}), suppose $n$ is even and $n_j=n_{j'}=2a$ for some $a\in \Z$. Let $f_*:V(C_n)\to\{1,\ldots,k\}$ be a $k$-color weak maximizer of $W$ with $\type(f_*)=\type(f)$ such that 
\[f_*^{-1}(j)=\left[0,a-1\right] \cup\left[\frac{n}{2},\frac{n}{2}+a-1\right]\]
and
\[f_*^{-1}(j')=\left[a,2a-1\right] \cup\left[\frac{n}{2}+a,\frac{n}{2}+2a-1\right]\]
(see Figure \ref{figure_n_even_n_i_equal_n_j}).
Since $f_*$ and $f$ are $k$-color weak maximizers with the same type we have $W(f)=W(f_*)$. Let $\hat{v}$ be the vertex $\frac{n}{2}+a-1$. Let $f_{**}$ be the coloring obtained from $f_*$ by changing the color of $\hat{v}$ from $j$ to $j'$. Since all of the color classes of $f_{**}$ are good, it follows by Theorem \ref{theorem_weak_max_on_C_n} that $f_{**}$ is a $k$-color weak maximizer with $\type(f_{**})=\type(f')$ and hence $W(f_{**})=W(f')$. Furthermore, the quantity $W(f_{**})-W(f_*)$ equals the sum of the distances from $\hat{v}$ to the vertices of $A_{j'}$ minus the sum of the distances from $\hat{v}$ to the vertices of $A_j\setminus\{\hat{v}\}$. In Figure \ref{figure_n_even_n_i_equal_n_j}, vertices of $f^{-1}(j)$ and $f^{-1}(j')$, other than $\hat{v}$, are labeled with their distances to $\hat{v}$. Therefore, from Figure \ref{figure_n_even_n_i_equal_n_j}, we see that
\[W(f_{**})-W(f_*)=\sum_{i=1}^a i+\sum_{i=1}^a(d-i) -\left[\sum_{i=1}^{a-1}i+\sum_{i=0}^{a-1}(d-i)\right]=ad-ad=0\]
and hence
\[W(f)=W(f_*)=W(f_{**})=W(f').\]

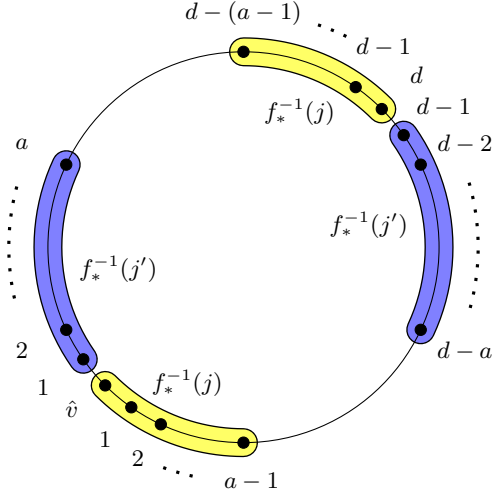
\begin{figure}
\begin{tikzpicture}[scale=1.3]
{\small
\def\radius{2}
\draw (0,0) circle (\radius);

\definecolor{r}{rgb}{1, 0.4, 0.4}    
\definecolor{b}{rgb}{0.5, 0.5, 1}    
\definecolor{yell}{rgb}{1, 1, 0.4}

  \begin{pgfonlayer}{background}
    
    \draw [black, line width=11pt, line cap=round, draw opacity=1,domain=45:90] plot ({\radius*cos(\x)}, {\radius*sin(\x)});  
  
    \draw [yell, line width=10pt, line cap=round, draw opacity=1,domain=45:90] plot ({\radius*cos(\x)}, {\radius*sin(\x)});
  
    \node[] at ({67.5}:\radius-0.5) {$f_*^{-1}(j)$};

    \draw [black,line width=11pt, line cap=round, draw opacity=1,domain=225:270] plot ({\radius*cos(\x)}, {\radius*sin(\x)});

    \draw [yell,line width=10pt, line cap=round, draw opacity=1,domain=225:270] plot ({\radius*cos(\x)}, {\radius*sin(\x)});
  
    \node[] at ({247.5}:\radius-0.5) {$f_*^{-1}(j)$};

    \node[draw, label={[label distance=5pt]200:$\hat{v}$}, circle, fill=black, inner sep=1.5pt] at ({225}:\radius) {};

    \node[draw, label={[label distance=5pt]235:$1$}, circle, fill=black, inner sep=1.5pt] at ({235}:\radius) {};

    \node[draw, label={[label distance=5pt]245:$2$}, circle, fill=black, inner sep=1.5pt] at ({245}:\radius) {};

    \draw [loosely dotted, very thick, domain=250:260] plot ({(\radius+0.4)*cos(\x)}, {(\radius+0.4)*sin(\x)});

    \node[draw, circle, fill=black, inner sep=1.5pt] at ({270}:\radius) {};

    \node[] at ({272}:\radius+0.4) {$a-1$};

    \node[draw, circle, fill=black, inner sep=1.5pt] at ({45}:\radius) {};

    \node[] at ({45}:\radius+0.5) {$d$};

    \node[draw, circle, fill=black, inner sep=1.5pt] at ({55}:\radius) {};

    \node[] at ({55}:\radius+0.5) {$d-1$};

    \draw [loosely dotted, very thick, domain=64:72] plot ({(\radius+0.4)*cos(\x)}, {(\radius+0.4)*sin(\x)});

    \node[draw, circle, fill=black, inner sep=1.5pt] at ({90}:\radius) {};

    \node[] at ({90}:\radius+0.4) {$d-(a-1)$};

    \draw [black,line width=11pt, line cap=round, draw opacity=1,domain=-25:35] plot ({\radius*cos(\x)}, {\radius*sin(\x)});

    \draw [b,line width=10pt, line cap=round, draw opacity=1,domain=-25:35] plot ({\radius*cos(\x)}, {\radius*sin(\x)});

    \node[] at ({10}:\radius-0.7) {$f_*^{-1}(j')$};

    \draw [black,line width=11pt, line cap=round, draw opacity=1,domain=-25+180:35+180] plot ({\radius*cos(\x)}, {\radius*sin(\x)});

    \draw [b,line width=10pt, line cap=round, draw opacity=1,domain=-25+180:35+180] plot ({\radius*cos(\x)}, {\radius*sin(\x)});

    \node[] at ({180+10}:\radius-0.7) {$f_*^{-1}(j')$};

    \node[draw, circle, fill=black, inner sep=1.5pt] at ({215}:\radius) {};

    \node[draw, circle, fill=black, inner sep=1.5pt] at ({205}:\radius) {};

    \node[draw, circle, fill=black, inner sep=1.5pt] at ({155}:\radius) {};

    \node[] at ({215}:\radius+0.5) {$1$};

    \node[] at ({205}:\radius+0.5) {$2$};

    \node[] at ({155}:\radius+0.5) {$a$};

    \node[draw, circle, fill=black, inner sep=1.5pt] at ({215+180}:\radius) {};

    \node[draw, circle, fill=black, inner sep=1.5pt] at ({205+180}:\radius) {};

    \node[draw, circle, fill=black, inner sep=1.5pt] at ({155+180}:\radius) {};

    \node[] at ({215+180}:\radius+0.5) {$d-1$};

    \node[] at ({205+180}:\radius+0.5) {$d-2$};

    \node[] at ({155+180}:\radius+0.5) {$d-a$};

    \draw [loosely dotted, very thick, domain=155+180+10:215+180-17] plot ({(\radius+0.4)*cos(\x)}, {(\radius+0.4)*sin(\x)});


    \draw [loosely dotted, very thick, domain=165:195] plot ({(\radius+0.4)*cos(\x)}, {(\radius+0.4)*sin(\x)});
    
  \end{pgfonlayer}
}
\end{tikzpicture}

\caption{\small Distances from $\hat{v}$ to vertices of color $j$ and $j'$ when $n$ is even and $n_j=n_{j'}=2a$, where $a\in\Z$.}\label{figure_n_even_n_i_equal_n_j}

\end{figure}

\begin{figure}
\begin{tikzpicture}[scale=1.3]
{\small
\def\radius{2}
\def\rot{125}
\draw (0,0) circle (\radius);

\definecolor{red}{rgb}{1, 0.4, 0.4}    
\definecolor{blue}{rgb}{0.5, 0.5, 1}    
\definecolor{yellow}{rgb}{1, 1, 0.4}

  \begin{pgfonlayer}{background}
    
    \draw [black, line width=11pt, line cap=round, draw opacity=1,domain=45:90] plot ({\radius*cos(\x)}, {\radius*sin(\x)});  
  
    \draw [yellow, line width=10pt, line cap=round, draw opacity=1,domain=45:90] plot ({\radius*cos(\x)}, {\radius*sin(\x)});
  
    \node[] at ({67.5}:\radius-0.5) {$f_*^{-1}(j)$};

    \draw [black,line width=11pt, line cap=round, draw opacity=1,domain=225:270] plot ({\radius*cos(\x)}, {\radius*sin(\x)});

    \draw [yellow,line width=10pt, line cap=round, draw opacity=1,domain=225:270] plot ({\radius*cos(\x)}, {\radius*sin(\x)});
  
    \node[] at ({247.5}:\radius-0.5) {$f_*^{-1}(j)$};

    \node[draw, circle, fill=black, inner sep=1.5pt] at ({225}:\radius) {};

    \node[] at ({220}:\radius+0.7) {$\ceil{\frac{n_{j}}{2}}-1$};
    
    \node[draw, label={}, circle, fill=black, inner sep=1.5pt] at ({250}:\radius) {};

    \node[] at ({250}:\radius+0.4) {$2$};
    
    \node[draw, label={}, circle, fill=black, inner sep=1.5pt] at ({260}:\radius) {};

    \node[] at ({260}:\radius+0.4) {$1$};
    
    \draw [loosely dotted, very thick, domain=233:243] plot ({(\radius+0.4)*cos(\x)}, {(\radius+0.4)*sin(\x)});

    \node[draw, circle, fill=black, inner sep=1.5pt] at ({270}:\radius) {};

    \node[] at ({270}:\radius+0.4) {$\hat{v}$};

    \node[draw, circle, fill=black, inner sep=1.5pt] at ({45}:\radius) {};

    \node[] at ({33}:\radius+1) {$d-\left(\floor{\frac{n_j}{2}}-1\right)$};

    \draw [loosely dotted, very thick, domain=55:65] plot ({(\radius+0.4)*cos(\x)}, {(\radius+0.4)*sin(\x)});
    
    \node[draw, circle, fill=black, inner sep=1.5pt] at ({80}:\radius) {};

    \node[] at ({75}:\radius+0.4) {$d-1$};

    \node[draw, circle, fill=black, inner sep=1.5pt] at ({90}:\radius) {};

    \node[] at ({90}:\radius+0.4) {$d$};

    \draw [black,line width=11pt, line cap=round, draw opacity=1,domain=-25+\rot:35+\rot] plot ({\radius*cos(\x)}, {\radius*sin(\x)});

    \draw [blue,line width=10pt, line cap=round, draw opacity=1,domain=-25+\rot:35+\rot] plot ({\radius*cos(\x)}, {\radius*sin(\x)});

    \node[] at ({\rot}:\radius-0.5) {$f_*^{-1}(j')$};

    \draw [black,line width=11pt, line cap=round, draw opacity=1,domain=-25+180+\rot:35+180+\rot] plot ({\radius*cos(\x)}, {\radius*sin(\x)});

    \draw [blue,line width=10pt, line cap=round, draw opacity=1,domain=-25+180+\rot:35+180+\rot] plot ({\radius*cos(\x)}, {\radius*sin(\x)});

    \node[] at ({180+\rot}:\radius-0.5) {$f_*^{-1}(j')$};

    \node[draw, circle, fill=black, inner sep=1.5pt] at ({215+\rot}:\radius) {};

    \node[draw, circle, fill=black, inner sep=1.5pt] at ({165+\rot}:\radius) {};

    \node[draw, circle, fill=black, inner sep=1.5pt] at ({155+\rot}:\radius) {};

    \node[] at ({215+\rot}:\radius+0.6) {$\floor{\frac{n_{j'}}{2}}$};

    \node[] at ({165+\rot}:\radius+0.4) {$2$};

    \node[] at ({155+\rot}:\radius+0.4) {$1$};

    \draw [loosely dotted, very thick, domain=175+\rot:215+\rot-15] plot ({(\radius+0.4)*cos(\x)}, {(\radius+0.4)*sin(\x)});

    \node[draw, circle, fill=black, inner sep=1.5pt] at ({215+180+\rot}:\radius) {};

    \node[draw, circle, fill=black, inner sep=1.5pt] at ({165+180+\rot}:\radius) {};

    \node[draw, circle, fill=black, inner sep=1.5pt] at ({155+180+\rot}:\radius) {};

    \node[] at ({215+180+\rot}:\radius+1.2) {$d-\left(\ceil{\frac{n_{j'}}{2}}-1\right)$};

    \node[] at ({165+180+\rot+5}:\radius+0.4) {$d-1$};

    \node[] at ({155+180+\rot}:\radius+0.4) {$d$};

    \draw [loosely dotted, very thick, domain=165+180+\rot+15:215+180+\rot-15] plot ({(\radius+0.4)*cos(\x)}, {(\radius+0.4)*sin(\x)});


  \end{pgfonlayer}
}
\end{tikzpicture}

\caption{\small Distances from $\hat{v}$ to vertices of color $j$ and $j'$ when $n$ is odd.}\label{figure_n_odd}

\end{figure}
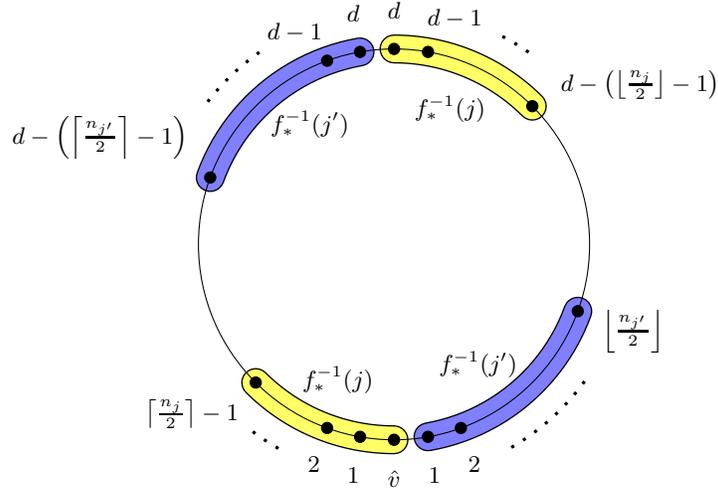

Now let us prove (\ref{item_main_case}). Since we are not in the case outlined in (\ref{item_weird_case}), we know that either $n$ is odd or $n_j$ and $n_{j'}$ are not equal and even. We will consider cases based on the parities of $n$, $n_j$ and $n_{j'}$. In each case (except for that mentioned in (\ref{item_weird_case})), we will prove the following.
\begin{itemize}
    \item[($*$)] There is a $k$-color weak maximizer $f_*:V(C_n)\to\{1,\ldots,k\}$ with $\type(f_*)=\type(f)$ and a vertex $\hat{v}\in f_*^{-1}(j)$ such that by changing the color of $\hat{v}$ to be $j'$, we obtain a new coloring $f_{**}$ with $W(f_{**})>W(f_*)$.
\end{itemize}
This will suffice because $\type(f_{**})=\type(f')$, and since $f$, $f_*$ and $f'$ are $k$-color weak maximizers we have $W(f')\geq W(f_{**})>W(f_*)=W(f)$.

Suppose $n$ is odd. Let $f_*:V(C_n)\to\{1,\ldots,k\}$ be a $k$-color weak maximizer of $W$ with $\type(f_*)=\type(f)$ such that
\[f_*^{-1}(j)=\left[0,\floor{\frac{n_j}{2}}-1\right]\cup\left[\floor{\frac{n}{2}},\floor{\frac{n}{2}}+\ceil{\frac{n_j}{2}}-1\right]\]
and 
\[f_*^{-1}(j')=\left[\floor{\frac{n}{2}}-\floor{\frac{n_{j'}}{2}}, \floor{\frac{n}{2}}-1\right]\cup\left[n-\ceil{\frac{n_{j'}}{2}},n-1\right]\]
(see Figure \ref{figure_n_odd}).
Let $\hat{v}$ be the vertex $\floor{\frac{n}{2}}$ and let $f_{**}$ be the coloring obtained from $f_*$ by changing the color of $\hat{v}$ from $j$ to $j'$. It follows that $\type(f_{**})=\type(f')$. Furthermore, the quantity $W(f_{**})-W(f_*)$ equals the sum of the distances from $\hat{v}$ to the vertices of $f_*^{-1}(j')$ minus the sum of the distances from $\hat{v}$ to the other vertices of $f_*^{-1}(j)$. In Figure \ref{figure_n_odd}, vertices of $f^{-1}(j)$ and $f^{-1}(j')$, other than $\hat{v}$, are labeled with their distances to $\hat{v}$. Therefore, from Figure \ref{figure_n_odd}, we see that

\begin{align}
W(f_{**})-W(f_*)=\sum_{i=0}^{\ceil{\frac{n_{j'}}{2}}-1}(d-i)+\sum_{i=1}^{\floor{\frac{n_{j'}}{2}}}i-\left[\sum_{i=0}^{\floor{\frac{n_j}{2}}-1}(d-i)+\sum_{i=1}^{\ceil{\frac{n_j}{2}}-1}i\right].\label{equation_delta_wiener1}
\end{align}

Let us consider (\ref{equation_delta_wiener1}) under various parity assumptions on $n_j$ and $n_{j'}$. Suppose $n_{j'}$ is even. If $n_j$ is even we have
\[W(f_{**})-W(f_*)=\frac{n_{j'}}{2}(d+1)-\frac{n_j}{2}d>0,\]
and if $n_j$ is odd we have
\begin{align*}
W(f_{**})-W(f_*)&=\frac{n_{j'}}{2}(d+1)-\left[\frac{n_j-1}{2}d+\frac{n_j+1}{2}-1\right]\\
    &=\frac{n_{j'}}{2}(d+1)-\frac{n_j-1}{2}(d+1)\\
    &>0;
\end{align*}
in both cases, the final inequality holds because $n_j\leq n_{j'}$.

Suppose $n_{j'}$ is odd. If $n_j$ is even then
\[W(f_{**})-W(f_*)=\frac{n_{j'}+1}{2}d-\frac{n_j}{2}d>0.\]
On the other hand, if $n_j$ is odd we have
\begin{align*}
W(f_{**})-W(f_*)&=\frac{n_{j'}+1}{2}d-\frac{n_j-1}{2}(d+1)\\
    &=\left(\frac{n_{j'}-n_j}{2}+1\right)d-\frac{n_j-1}{2}\\
    &\geq\floor{\frac{n}{2}}-\floor{\frac{n_j}{2}}\\
    &>0,
\end{align*}
where the last inequality holds because $n_j$ and $n$ are odd and $n_j\neq n$. This completes the proof of (2) in the case where $n$ is odd.

\begin{figure}
\begin{tikzpicture}[scale=1.3]
{\small
\def\radius{2}
\def\rot{125}
\draw (0,0) circle (\radius);

\definecolor{red}{rgb}{1, 0.4, 0.4}    
\definecolor{blue}{rgb}{0.5, 0.5, 1}    
\definecolor{yellow}{rgb}{1, 1, 0.4}

  \begin{pgfonlayer}{background}
    
    \draw [black, line width=11pt, line cap=round, draw opacity=1,domain=45:90] plot ({\radius*cos(\x)}, {\radius*sin(\x)});  
  
    \draw [yellow, line width=10pt, line cap=round, draw opacity=1,domain=45:90] plot ({\radius*cos(\x)}, {\radius*sin(\x)});
  
    \node[] at ({67.5}:\radius-0.5) {$f_*^{-1}(j)$};

    \draw [black,line width=11pt, line cap=round, draw opacity=1,domain=225:270] plot ({\radius*cos(\x)}, {\radius*sin(\x)});

    \draw [yellow,line width=10pt, line cap=round, draw opacity=1,domain=225:270] plot ({\radius*cos(\x)}, {\radius*sin(\x)});
  
    \node[] at ({247.5}:\radius-0.5) {$f_*^{-1}(j)$};

    \node[draw, circle, fill=black, inner sep=1.5pt] at ({225}:\radius) {};

    \node[] at ({220}:\radius+0.7) {$\ceil{\frac{n_{j}}{2}}-1$};
    
    \node[draw, label={}, circle, fill=black, inner sep=1.5pt] at ({250}:\radius) {};

    \node[] at ({250}:\radius+0.4) {$2$};
    
    \node[draw, label={}, circle, fill=black, inner sep=1.5pt] at ({260}:\radius) {};

    \node[] at ({260}:\radius+0.4) {$1$};
    
    \draw [loosely dotted, very thick, domain=233:243] plot ({(\radius+0.4)*cos(\x)}, {(\radius+0.4)*sin(\x)});

    \node[draw, circle, fill=black, inner sep=1.5pt] at ({270}:\radius) {};

    \node[] at ({270}:\radius+0.4) {$\hat{v}$};

    \node[draw, circle, fill=black, inner sep=1.5pt] at ({45}:\radius) {};

    \node[] at ({33}:\radius+1) {$d-\left(\floor{\frac{n_j}{2}}-1\right)$};

    \draw [loosely dotted, very thick, domain=55:65] plot ({(\radius+0.4)*cos(\x)}, {(\radius+0.4)*sin(\x)});
    
    \node[draw, circle, fill=black, inner sep=1.5pt] at ({80}:\radius) {};

    \node[] at ({75}:\radius+0.4) {$d-1$};

    \node[draw, circle, fill=black, inner sep=1.5pt] at ({90}:\radius) {};

    \node[] at ({90}:\radius+0.4) {$d$};

    \draw [black,line width=11pt, line cap=round, draw opacity=1,domain=-25+\rot:35+\rot] plot ({\radius*cos(\x)}, {\radius*sin(\x)});

    \draw [blue,line width=10pt, line cap=round, draw opacity=1,domain=-25+\rot:35+\rot] plot ({\radius*cos(\x)}, {\radius*sin(\x)});

    \node[] at ({\rot}:\radius-0.5) {$f_*^{1}(j')$};

    \draw [black,line width=11pt, line cap=round, draw opacity=1,domain=-25+180+\rot:35+180+\rot] plot ({\radius*cos(\x)}, {\radius*sin(\x)});

    \draw [blue,line width=10pt, line cap=round, draw opacity=1,domain=-25+180+\rot:35+180+\rot] plot ({\radius*cos(\x)}, {\radius*sin(\x)});

    \node[] at ({180+\rot}:\radius-0.5) {$f_*^{-1}(j')$};

    \node[draw, circle, fill=black, inner sep=1.5pt] at ({215+\rot}:\radius) {};

    \node[draw, circle, fill=black, inner sep=1.5pt] at ({165+\rot}:\radius) {};

    \node[draw, circle, fill=black, inner sep=1.5pt] at ({155+\rot}:\radius) {};

    \node[] at ({215+\rot}:\radius+0.6) {$\floor{\frac{n_{j'}}{2}}$};

    \node[] at ({165+\rot}:\radius+0.4) {$2$};

    \node[] at ({155+\rot}:\radius+0.4) {$1$};

    \draw [loosely dotted, very thick, domain=175+\rot:215+\rot-15] plot ({(\radius+0.4)*cos(\x)}, {(\radius+0.4)*sin(\x)});

    \node[draw, circle, fill=black, inner sep=1.5pt] at ({215+180+\rot}:\radius) {};

    \node[draw, circle, fill=black, inner sep=1.5pt] at ({165+180+\rot}:\radius) {};

    \node[draw, circle, fill=black, inner sep=1.5pt] at ({155+180+\rot}:\radius) {};

    \node[] at ({215+180+\rot}:\radius+1.2) {$d-\ceil{\frac{n_{j'}}{2}}$};

    \node[] at ({165+180+\rot+10}:\radius+0.5) {$d-2$};

    \node[] at ({155+180+\rot+2}:\radius+0.4) {$d-1$};

    \draw [loosely dotted, very thick, domain=165+180+\rot+20:215+180+\rot-15] plot ({(\radius+0.4)*cos(\x)}, {(\radius+0.4)*sin(\x)});


  \end{pgfonlayer}
}
\end{tikzpicture}

\caption{\small Distances from $\hat{v}$ to vertices of color $j$ and $j'$ when $n$ is even and $n_j$ is even.}\label{figure_n_even_n_j_even}

\end{figure}

To continue our proof of (2), suppose $n$ is even. Further suppose that $n_j$ is even. Let $f_*:V(C_n)\to\{1,\ldots,k\}$ be a $k$-color weak maximizer of $W$ with $\type(f_*)=\type(f)$ such that
\[f^{-1}_*(j)=\left[0,\frac{n_j}{2}-1\right]\cup\left[\frac{n}{2},\frac{n}{2}+\frac{n_j}{2}-1\right]\]
and
\[f_*^{-1}(j')=\left[\frac{n}{2}-\floor{\frac{n_{j'}}{2}},\frac{n}{2}-1\right]\cup\left[n-\ceil{\frac{n_{j'}}{2}},n-1\right].\]
Let $\hat{v}\in f_*^{-1}(j)$ be the vertex $\frac{n}{2}$. In Figure \ref{figure_n_even_n_j_even}, vertices of $f^{-1}(j)$ and $f^{-1}(j')$, other than $\hat{v}$, are labeled with their distances to $\hat{v}$. Let $f_{**}$ be the coloring obtained from $f_*$ by changing the color of $\hat{v}$ from $j$ to $j'$. It follows that 
$\type(f_{**})=\type(f')$. As in previous cases, using Figure \ref{figure_n_even_n_j_even}, we see that
\begin{align*}
W(f_{**})-W(f_*)&=\sum_{i=1}^{\ceil{\frac{n_{j'}}{2}}}(d-i)+\sum_{i=1}^{\floor{\frac{n_{j'}}{2}}}i-\left[\sum_{i=0}^{\frac{n_j}{2}-1}(d-i)+\sum_{i=1}^{\frac{n_j}{2}-1}i\right].
\end{align*}
Suppose $n_{j'}$ is even. Since we are not in the case outlined in (1), it follows that $n_j<n_{j'}$ and hence
\[W(f_{**})-W(f_*)=\frac{n_{j'}}{2}d-\frac{n_j}{2}d>0.\]
Suppose $n_{j'}$ is odd. Then we have
\begin{align*}
    W(f_{**})-W(f_*)&=\frac{n_{j'}+1}{2}d-\frac{n_{j'}+1}{2}-\frac{n_j}{2}d.
\end{align*}
It suffices to show that 
\[(n_{j'}-n_j+1)d>n_{j'}+1,\]
but this follows easily from the fact that $n_{j'}-n_j\geq 1$. Thus $W(f_{**})-W(f_*)>0$.

{\small

\begin{figure}
\begin{tikzpicture}[scale=1.3]
{\small
\def\radius{2}
\draw (0,0) circle (\radius);

\definecolor{red}{rgb}{1, 0.4, 0.4}    
\definecolor{blue}{rgb}{0.5, 0.5, 1}    
\definecolor{yellow}{rgb}{1, 1, 0.4}

  \begin{pgfonlayer}{background}
    
    \draw [black, line width=11pt, line cap=round, draw opacity=1,domain=45:90] plot ({\radius*cos(\x)}, {\radius*sin(\x)});  
  
    \draw [yellow, line width=10pt, line cap=round, draw opacity=1,domain=45:90] plot ({\radius*cos(\x)}, {\radius*sin(\x)});
  
    \node[] at ({67.5}:\radius-0.5) {$f_*^{-1}(j)$};

    \draw [black,line width=11pt, line cap=round, draw opacity=1,domain=225:270] plot ({\radius*cos(\x)}, {\radius*sin(\x)});

    \draw [yellow,line width=10pt, line cap=round, draw opacity=1,domain=225:270] plot ({\radius*cos(\x)}, {\radius*sin(\x)});
  
    \node[] at ({247.5}:\radius-0.5) {$f_*^{-1}(j)$};

    \node[draw, label={[label distance=5pt]200:$\hat{v}$}, circle, fill=black, inner sep=1.5pt] at ({225}:\radius) {};

    \node[draw, label={[label distance=5pt]235:$1$}, circle, fill=black, inner sep=1.5pt] at ({235}:\radius) {};

    \node[draw, label={[label distance=5pt]245:$2$}, circle, fill=black, inner sep=1.5pt] at ({245}:\radius) {};

    \draw [loosely dotted, very thick, domain=250:260] plot ({(\radius+0.4)*cos(\x)}, {(\radius+0.4)*sin(\x)});

    \node[draw, circle, fill=black, inner sep=1.5pt] at ({270}:\radius) {};

    \node[] at ({272}:\radius+0.4) {$\ceil{\frac{n_j}{2}}-1$};

    \node[draw, circle, fill=black, inner sep=1.5pt] at ({45}:\radius) {};

    \node[] at ({45}:\radius+0.5) {$d-1$};

    \node[draw, circle, fill=black, inner sep=1.5pt] at ({55}:\radius) {};

    \node[] at ({55}:\radius+0.5) {$d-2$};

    \draw [loosely dotted, very thick, domain=64:75] plot ({(\radius+0.4)*cos(\x)}, {(\radius+0.4)*sin(\x)});

    \node[draw, circle, fill=black, inner sep=1.5pt] at ({90}:\radius) {};

    \node[] at ({90}:\radius+0.4) {$d-\floor{\frac{n_j}{2}}$};

    \draw [black,line width=11pt, line cap=round, draw opacity=1,domain=-25:35] plot ({\radius*cos(\x)}, {\radius*sin(\x)});

    \draw [blue,line width=10pt, line cap=round, draw opacity=1,domain=-25:35] plot ({\radius*cos(\x)}, {\radius*sin(\x)});

    \node[] at ({10}:\radius-0.7) {$f_*^{-1}(j')$};

    \draw [black,line width=11pt, line cap=round, draw opacity=1,domain=-25+180:35+180] plot ({\radius*cos(\x)}, {\radius*sin(\x)});

    \draw [blue,line width=10pt, line cap=round, draw opacity=1,domain=-25+180:35+180] plot ({\radius*cos(\x)}, {\radius*sin(\x)});

    \node[] at ({180+10}:\radius-0.7) {$f_*^{-1}(j')$};

    \node[draw, circle, fill=black, inner sep=1.5pt] at ({215}:\radius) {};

    \node[draw, circle, fill=black, inner sep=1.5pt] at ({205}:\radius) {};

    \node[draw, circle, fill=black, inner sep=1.5pt] at ({155}:\radius) {};

    \node[] at ({215}:\radius+0.5) {$1$};

    \node[] at ({205}:\radius+0.5) {$2$};

    \node[] at ({155}:\radius+0.5) {$\floor{\frac{n_{j'}}{2}}$};

    \node[draw, circle, fill=black, inner sep=1.5pt] at ({215+180}:\radius) {};

    \node[draw, circle, fill=black, inner sep=1.5pt] at ({205+180}:\radius) {};

    \node[draw, circle, fill=black, inner sep=1.5pt] at ({155+180}:\radius) {};

    \node[] at ({215+180}:\radius+0.4) {$d$};

    \node[] at ({205+180}:\radius+0.5) {$d-1$};

    \node[] at ({155+185}:\radius+1.1) {$d-\left(\ceil{\frac{n_{j'}}{2}}-1\right)$};

    \draw [loosely dotted, very thick, domain=155+180+10:215+180-20] plot ({(\radius+0.4)*cos(\x)}, {(\radius+0.4)*sin(\x)});


    \draw [loosely dotted, very thick, domain=165:195] plot ({(\radius+0.4)*cos(\x)}, {(\radius+0.4)*sin(\x)});
    
  \end{pgfonlayer}
}
\end{tikzpicture}

\caption{\small Distances from $\hat{v}$ to vertices of color $j$ and $j'$ when $n$ is even and $n_j$ is odd.}\label{figure_n_even_n_j_odd}

\end{figure}

}

To finish the proof of (2) it only remains to handle the case in which $n$ is even and $n_j$ is odd. Let $f_*:V(C_n)\to\{1,\ldots,k\}$ be a $k$-color weak maximizer of $W$ with $\type(f_*)=\type(f)$ such that 
\[f_*^{-1}(j)=\left[0,\floor{\frac{n_j}{2}}-1\right]\cup\left[\frac{n}{2},\frac{n}{2}+\ceil{\frac{n_j}{2}}-1\right]\]
and
\[f_*^{-1}(j')=\left[\floor{\frac{n_j}{2}},\floor{\frac{n_j}{2}}+\ceil{\frac{n_{j'}}{2}}-1\right]\cup\left[\frac{n}{2}+\ceil{\frac{n_j}{2}},\frac{n}{2}+\ceil{\frac{n_j}{2}}+\floor{\frac{n_{j'}}{2}}-1\right].\]
Let $\hat{v}\in f_*^{-1}(j)$ be the vertex $\hat{v}=\frac{n}{2}+\ceil{\frac{n_j}{2}}-1$. In Figure \ref{figure_n_even_n_j_odd}, vertices of $f_*^{-1}(j)$ and $f_*^{-1}(j')$ are labeled with their distances to $\hat{v}$. We let $f_{**}$ be the coloring obtained from $f_*$ by changing the color of $\hat{v}$ from $j$ to $j'$. As in previous cases, using Figure \ref{figure_n_even_n_j_odd}, we obtain
\[W(f_{**})-W(f_*)=\sum_{i=0}^{\ceil{\frac{n_{j'}}{2}}-1}(d-i)+\sum_{i=1}^{\floor{\frac{n_{j'}}{2}}}i-\left[\sum_{i=1}^{\floor{\frac{n_j}{2}}}(d-i)+\sum_{i=1}^{\ceil{\frac{n_j}{2}}-1}i\right].\]
If $n_{j'}$ is even we have
\[W(f_{**})-W(f_*)=\ceil{\frac{n_{j'}}{2}}d+\frac{n_{j'}}{2}-\floor{\frac{n_j}{2}}d>0\]
since $n_j<n_{j'}$, and if $n_{j'}$ is odd we have
\[W(f_{**})-W(f_*)=\ceil{\frac{n_{j'}}{2}}d-\floor{\frac{n_j}{2}}d>0.\]
This completes the proof of (2).\end{proof}

\begin{figure}
\centering
\begin{tikzpicture}[scale=0.4]
    \definecolor{r}{rgb}{1, 0.4, 0.4}    
    \definecolor{b}{rgb}{0.4, 0.4, 1}    
    \definecolor{y}{rgb}{1, 1, 0}
  
  \tikzset{
    ynode/.style={draw, fill=y, circle, minimum size=2mm, inner sep=0pt},
    rnode/.style={draw, fill=r, circle, minimum size=2mm, inner sep=0pt},
    bnode/.style={draw, fill=b, circle, minimum size=2mm, inner sep=0pt}
  }

  \newcommand{\cycle}[6]{
    \foreach \i in {1,...,#3} {
      \coordinate (p\i) at ({#1 + cos((360/#3)*(\i-1)+90)}, {#2 + sin((360/#3)*(\i-1)+90)});
    }

    \foreach \i [count=\j] in {1,...,#3} {
      \ifnum\j=#3
        \draw (p\j) -- (p1);
      \else
        \draw (p\j) -- (p\the\numexpr\j+1\relax);
      \fi
    }

    \foreach \x in #4 {
        \node[ynode] at (p\x) {};
    }
    \foreach \x in #5 {
        \node[rnode] at (p\x) {};
    }
    \foreach \x in #6 {
        \node[bnode] at (p\x) {};
    }
    
  }

\def\a{-4}
\def\b{3}
\def\c{-3}

\newcommand{\upeq}{\mathbin{\rotatebox[origin=c]{90}{$=$}}}

\newcommand{\upl}{\mathbin{\rotatebox[origin=c]{90}{$<$}}}

\newcommand{\upmaj}{\mathbin{\rotatebox[origin=c]{90}{$\maj$}}}

\node[] at (\a,3) {type};

\node[] at (-0.2,3) {coloring};

\node[] at (\b+0.5,4) {Wiener};
\node[] at (\b+0.5,3) {index};
\draw (\a-2,2)--(\b+1.5,2);

\node[] at (\a,0) {$(1,1,6)$};
\node[] at (\a,\c) {$(1,2,5)$};
\node[] at (\a,2*\c) {$(1,3,4)$};
\node[] at (\a,3*\c) {$(2,2,4)$};
\node[] at (\a,4*\c) {$(2,3,3)$};

\node[] at (\a,0.5*\c) {$\upmaj$};
\node[] at (\a,1.5*\c) {$\upmaj$};
\node[] at (\a,2.5*\c) {$\upmaj$};
\node[] at (\a,3.5*\c) {$\upmaj$};

\node[] at (\b,0) {$36$};
\node[] at (\b,\c) {$28$};
\node[] at (\b,2*\c) {$24$};
\node[] at (\b,3*\c) {$24$};
\node[] at (\b,4*\c) {$20$};

\node[] at (\b,0.5*\c) {$\upl$};
\node[] at (\b,1.5*\c) {$\upl$};
\node[] at (\b,2.5*\c) {$\upeq$};
\node[] at (\b,3.5*\c) {$\upl$};

\cycle{0}{0}{8}{{2,3,4,6,7,8}}{{5}}{{1}}
\cycle{0}{\c}{8}{{3,4,5,7,8}}{{2,6}}{{1}}
\cycle{0}{2*\c}{8}{{3,4,7,8}}{{2,5,6}}{{1}}
\cycle{0}{3*\c}{8}{{3,4,7,8}}{{2,6}}{{1,5}}
\cycle{0}{4*\c}{8}{{4,7,8}}{{2,3,6}}{{1,5}}

\draw[draw=black, rounded corners=5mm, dashed] (-2,-4.5) rectangle ++(4,-6);

\node[] at (0,-6) {$f'$};
\node[] at (0,-9) {$f$};

\end{tikzpicture}\hspace{5em}\begin{tikzpicture}[scale=0.4]
    \definecolor{r}{rgb}{1, 0.4, 0.4}    
    \definecolor{b}{rgb}{0.4, 0.4, 1}    
    \definecolor{y}{rgb}{1, 1, 0}
  
  \tikzset{
    ynode/.style={draw, fill=y, circle, minimum size=2mm, inner sep=0pt},
    rnode/.style={draw, fill=r, circle, minimum size=2mm, inner sep=0pt},
    bnode/.style={draw, fill=b, circle, minimum size=2mm, inner sep=0pt}
  }

  \newcommand{\cycle}[6]{
    \foreach \i in {1,...,#3} {
      \coordinate (p\i) at ({#1 + cos((360/#3)*(\i-1)+90)}, {#2 + sin((360/#3)*(\i-1)+90)});
    }

    \foreach \i [count=\j] in {1,...,#3} {
      \ifnum\j=#3
        \draw (p\j) -- (p1);
      \else
        \draw (p\j) -- (p\the\numexpr\j+1\relax);
      \fi
    }

    \foreach \x in #4 {
        \node[ynode] at (p\x) {};
    }
    \foreach \x in #5 {
        \node[rnode] at (p\x) {};
    }
    \foreach \x in #6 {
        \node[bnode] at (p\x) {};
    }
    
  }

\newcommand{\upeq}{\mathbin{\rotatebox[origin=c]{90}{$=$}}}

\newcommand{\upl}{\mathbin{\rotatebox[origin=c]{90}{$<$}}}

\newcommand{\upmaj}{\mathbin{\rotatebox[origin=c]{90}{$\maj$}}}

\def\a{-4}
\def\b{3}
\def\c{-3}

\node[] at (\a,3) {type};

\node[] at (-0.2,3) {coloring};

\node[] at (\b+0.5,4) {Wiener};
\node[] at (\b+0.5,3) {index};
\draw (\a-2,2)--(\b+1.5,2);

\node[] at (\a,0) {$(1,1,5)$};
\node[] at (\a,\c) {$(1,2,4)$};
\node[] at (\a,2*\c) {$(1,3,3)$};
\node[] at (\a,3*\c) {$(2,2,3)$};

\node[] at (\a,0.5*\c) {$\upmaj$};
\node[] at (\a,1.5*\c) {$\upmaj$};
\node[] at (\a,2.5*\c) {$\upmaj$};

\node[] at (\b,0) {$21$};
\node[] at (\b,\c) {$16$};
\node[] at (\b,2*\c) {$14$};
\node[] at (\b,3*\c) {$13$};

\node[] at (\b,0.5*\c) {$\upl$};
\node[] at (\b,1.5*\c) {$\upl$};
\node[] at (\b,2.5*\c) {$\upl$};

\cycle{0}{0}{7}{{2,3,5,6,7}}{{4}}{{1}}
\cycle{0}{\c}{7}{{3,4,6,7}}{{2,5}}{{1}}
\cycle{0}{2*\c}{7}{{3,4,7}}{{2,5,6}}{{1}}
\cycle{0}{3*\c}{7}{{3,6,7}}{{2,5}}{{1,4}}

\end{tikzpicture}
\caption{\small Less equitable $k$-color weak maximizers of $W$ on $C_n$ have larger Wiener index, unless we are in the case outlined in Theorem \ref{theorem_maj_cycles}(\ref{item_weird_case}), an example of which appears in the dashed box above. This is the case in which $n$ is even, for some $j,j'\in\{1,\ldots,k\}$ we have $|f^{-1}(j)|=|f^{-1}(j')|=2a$ where $a$ is a positive integer and $\type(f')=R_{j,j'}(\type(f))$.}
\label{figure_3cwm_maj}
\end{figure}

Let us formulate a version of Corollary \ref{corollary_largest_and_smallest_paths} for cycles. In order to do this we need to consider various cases that will be based on some new definitions.

\begin{definition} Suppose $k\leq n$ are positive integers and let $T\subseteq \Z^k_\uparrow$ be a collection of $k$-tuples of integers written in non-decreasing order. We define $\mathcal{R}(T)$ to be the collection of all $k$-tuples of integers $s$ written in non-decreasing order such that $s\in T$ or the following condition holds:
\begin{itemize}
\item $s=R_{j,j'}(t)$ where $t=(n_1,\ldots,n_k)\in T$, $1\leq j\leq j'\leq k$, and $n_j$ and $n_{j'}$ are even and equal.
\end{itemize}
Inductively, we define $\mathcal{R}^{\ell+1}(T)=\mathcal{R}(\mathcal{R}^\ell(T))$, for positive integers $\ell$. We let $\mathcal{R}^\infty(T)=\bigcup_{\ell=1}^\infty \mathcal{R}^\ell(T)$.
\end{definition}
Notice that, in general $T\subseteq\mathcal{R}(T)$ and if for every $k$-tuple $t=(n_1,\ldots,n_k)$ in $T$ there does not exist $j$ and $j'$ as above with $n_j$ and $n_j'$ equal and even, then $\mathcal{R}(T)=T$.

Notice that for $(a_1,\ldots,a_k)\in\Z^k_\uparrow$, if it is the case that for all $j,j'\in\{1,\ldots,k\}$ the entries $a_j$ and $a_{j'}$ are not equal and even, then $\mathcal{R}^\infty(\{(a_1,\ldots,a_k)\})=\{(a_1,\ldots,a_k)\}$. Whereas, on the other hand, if there exist $j,j'\in\{1,\ldots,k\}$ such that $a_j$ and $a_{j'}$ are equal and even, then $\mathcal{R}^\infty(\{(a_1,\ldots,a_k)\})$ will contain $k$-tuples other than $(a_1,\ldots, a_k)$.

We obtain the following corollaries from Theorem \ref{theorem_maj_cycles}.

\begin{corollary}\label{corollary_Weiner_index_R_infty}
If $f,g:V(C_n)\to\{1,\ldots,k\}$ are $k$-color weak maximizers of $W$ on $C_n$ and $f,g\in\mathcal{R}^\infty(\{t(n,k\})$ then $W(f)=W(g)$, where $t(n,k)$ is the $k$-tuple of positive integers defined in Lemma \ref{lemma_equitable_tuples}.
\end{corollary}

\begin{corollary}\label{corollary_W_does_not_go_down}
Suppose $f,g:V(C_n)\to\{1,\ldots,k\}$ are $k$-color weak maximizers of $W$ on $C_n$ and  $\type(g)\maj \type(f)$. Then $W(g)\leq W(f)$.
\end{corollary}

\begin{proof}
Since $\type(g)\maj\type(f)$, it follows by Lemma \ref{lemma_reverse_robin_hood} that $\type(f)$ can be obtained by applying a finite sequence of reverse Robin Hood transfers to $\type(g)$, producing a finite sequence of $k$-tuples, each of which is in $\Z^k_\uparrow$: \[\type(g)=t_0\maj t_1\maj\cdots\maj t_\ell=\type(f).\] 
For each $i\in\{1,\ldots,\ell-1\}$, choose a $k$-color weak maximizer $f_i$ of $W$ on $C_n$ with $\type(f_i)=t_i$.
Then, by Theorem \ref{theorem_maj_cycles}, it follows that
\[W(g)\leq W(f_1)\leq W(f_2)\leq\cdots\leq W(f_{\ell-1})\leq W(f).\]
\end{proof}

\begin{corollary}\label{corollary_cycles_smallest}
Suppose $2\leq k\leq n$ are integers and $f:V(C_n)\to\{1,\ldots,k\}$ is a $k$-color weak maximizer of $W$ on $C_n$.
\begin{enumerate}
\item Suppose $n$ is even and let $t(n,k)\in\Z^k_\uparrow$ be the $k$-tuple defined in Lemma \ref{lemma_equitable_tuples}. Then, $W(f)=\min\{W(f')\st f'\in\WM_k(C_n)\}$ if and only if $\type(f)\in\mathcal{R}^\infty( \{t(n,k)\})$.
\item Otherwise, if $n$ is odd then $W(f)=\min\{W(f')\st f'\in\WM_k(C_n)\}$ if and only if the number vertices in any two color classes of $f$ differ by at most one.
\end{enumerate}
\end{corollary}

\begin{proof}
Let us prove (1); the proof of (2) is similar, but easier. Suppose $f$ is a $k$-color weak maximizer on $C_n$ and $n$ is even. For the backward direction, suppose $\type(f)\in\mathcal{R}^\infty(\{t(n,k)\})$. Let $g$ be a $k$-color weak maximizer on $C_n$ with $\type(g)=t(n,k)$. By Lemma \ref{lemma_equitable_tuples} and Lemma \ref{corollary_W_does_not_go_down}, we see that $W(g)=\min\{W(f')\st f'\in\WM_k(C_n)\}$, and by Lemma \ref{corollary_Weiner_index_R_infty} we have $W(f)=W(g)$. For the forward direction of (1), suppose $\type(f)\notin\mathcal{R}^\infty(\{t(n,k)\})$. Let $g$ be a $k$-color weak maximizer on $C_n$ with $\type(g)=t(n,k)$. Then $\type(f)\neq\type(g)$ and $\type(g)\maj\type(f)$. By Lemma \ref{lemma_reverse_robin_hood}, $\type(f)$ can be obtained from $\type(g)$ by applying a finite sequence of reverse Robin Hood transfers, thus producing a sequence of $k$-tuples:
\[\type(g)=t_0\maj t_1\maj \cdots\maj t_\ell=\type(f).\] 
Since $\type(f)\notin\mathcal{R}^\infty(\{t(n,k\})$, it must be the case that for some $i\in\{1,\ldots,\ell\}$, $t_i$ is obtained from $t_{i-1}$ by a reverse Robin Hood transfer between two entries that are not even and equal. Fix $k$-color weak maximizers $h$ and $h'$ of $W$ on $C_n$ with $\type(h)=t_{i-1}$ and $\type(h')=t_i$. Then, by Theorem \ref{theorem_maj_cycles} and Lemma \ref{corollary_W_does_not_go_down}, we have 
\[W(g)\leq W(h)<W(h')\leq W(f).\]
Thus $W(f)\neq \min\{W(f')\st f'\in\WM_k(C_n)\}$.
\end{proof}

Now let us characterize which $k$-color weak maximizers on $C_n$ have the largest possible Wiener index.

\begin{corollary}\label{corollary_cycles_largest}
Suppose $2\leq k\leq n$ are integers and $f:V(C_n)\to\{1,\ldots,k\}$ is a $k$-color weak maximizer of $W$ on $C_n$.
\begin{enumerate}
\item Suppose $n$ is even and $n-k\leq 2$. Then $W(f)=\max\{W(f')\st f'\in\WM_k(C_n)\}$.
\item Otherwise, if $n$ is odd or if $n-k>2$, then $W(f)=\max\{W(f')\st f'\in\WM_k(C_n)\}$ if and only if $f$ has a color class of the largest possible size $n-(k-1)$.
\end{enumerate}
\end{corollary}

\begin{proof}
Suppose $2\leq k\leq n$ and fix any $k$-color weak maximizer $f$ of $W$ on $C_n$. Let us prove (1). Assume $n$ is even and $n-k\leq 2$. Suppose $n-k=0$. Then every $k$-color weak maximizer has type equal to $(1,\ldots,1)$, and thus has Wiener index equal to $0$. So, of course, $W(f)=\max\{W(f')\st f'\in\WM_k(C_n)\}$. 

Suppose $n-k=1$. Then every $k$-color weak maximizer of $W$ on $C_n$ has type equal to the $k$-tuple $(1,\ldots,1,2)$, and therefore has Wiener index equal to $n/2$. Hence $W(f)=\max\{W(f')\st f'\in\WM_k(C_n)\}$.

Suppose $n-k=2$. Then every $k$-color weak maximizer of $W$ on $C_n$ has type equal to a $k$-tuple of the form $(1,\ldots,1,2,2)$ or $(1,\ldots,1,3)$, and thus has Wiener index equal to $n$ (since $n$ is even). Hence $W(f)=\max\{W(f')\st f'\in\WM_k(C_n)\}$.

Now let us prove (2). Suppose $n$ is odd or $n-k>2$. For the reverse direction, suppose the $k$-color weak maximizer $f$ of $W$ on $C_n$ has a color class of the largest possible size $n-(k-1)$. So, $\type(f)$ is the $k$-tuple $(1,\ldots,1,n-(k-1))$. Let $g$ be any $k$-color weak maximizer of $W$ on $C_n$. Then $\type(g)\maj\type(f)$, which implies that $\type(f)$ can be obtained from $\type(g)$ by applying a finite sequence of reverse Robin Hood transfers, thus producing a sequence of $k$-tuples 
\[\type(g)=t_0\maj t_1\maj\cdots\maj t_\ell=\type(f).\]
For each $i\in\{1,\ldots,\ell-1\}$ choose a $k$-color weak maximizer $f_i$ with $\type(f_i)=t_i$. 
If $n$ is odd then by Theorem \ref{theorem_maj_cycles} we have
\[W(g)<W(f_1)<\cdots W(f_{\ell-1})<W(f).\]
If $n$ is not odd, then by assumption we may suppose that $n-k>2$. We see that $t_\ell=(1,\ldots,1,n-(k-1))$ and $t_{\ell-1}=(1,\ldots,1,2,n-k)$, and since $n-k>2$, the $k$-tuple $t_{\ell-1}$ does not have two entries that are equal and even. Hence, by Theorem \ref{theorem_maj_cycles} we have
\[W(g)\leq W(f_{\ell-1})<W(f).\]
Therefore, $W(f)=\max\{W(f')\st f'\in\WM_k(C_n)\}$

For the forward direction of (2), suppose $f$ does not have a color class of the largest possible size $n-(k-1)$. Let $g$ be a $k$-color weak maximizer of $W$ on $C_n$ that does have a color class of the largest possible size, that is, $\type(g)$ is the $k$-tuple $(1,\ldots,1,n-(k-1))$. Then, it follows that $\type(f)\maj \type(g)$ and $\type(f)\neq\type(g)$. This means that $\type(g)$ can be obtained from $\type(f)$ by applying a finite sequence of reverse Robin Hood transfers, thus producing a sequence of $k$-tuples:
\[\type(f)=t_0\maj t_1\maj\cdots\maj t_\ell=\type(g).\]
For each $i\in\{1,\ldots,\ell-1\}$ let $f_i$ be a $k$-color weak maximizer of $W$ on $C_n$ with $\type(f_i)=t_i$. Then, by an argument similar to that of the reverse direction of (2), given that $n$ is odd or $n-k>2$, it follows that 
$W(f)<W(f_{\ell-1})<W(g)$,
and thus $W(f)\neq\max\{W(f')\st f'\in\WM_k(C_n)\}$.
\end{proof}

\section{Questions}\label{section_questions}

We close the paper with several open questions. First, we state two rather broad classes of problems.

\begin{problem}\label{problem_find_maximizers}
Given a graph $G$, what are the $k$-color weak maximizers of $W$ on $G$? What are the $k$-color local weak maximizers of $W$ on $G$? What are the $k$-color maximizers of $W$ on $G$?
\end{problem}
For example, we think addressing Problem \ref{problem_find_maximizers} for various natural graphs such as toroidal graphs, grid graphs, powers of cycles, circulant graphs, corona graphs, Cayley graphs, etc. could lead to some approachable and nontrivial problems.

Recall that by Theorem \ref{theorem_weak_max_on_C_n}, a surjective function $f:V(C_n)\to\{1,\ldots,k\}$ is a $k$-color weak maximizer of $W$ on $C_n$ if and only if for all $i\in\{1,\ldots,k\}$ the set of vertices $f^{-1}(i)$ is a maximzer of $W$ as a set. On the other hand, there are $k$-color weak maximizers $f:V(P_n)\to\{1,\ldots,k\}$ of $W$ on $P_n$ whose color classes are not all maximizers of $W$ on $P_n$ as sets.

\begin{problem}
For what class of graphs $G$ will a surjective function $f:V(G)\to\{1,\ldots,k\}$ be a $k$-color weak maximizer of $W$ on $G$ if and only if all of the color classes $f^{-1}(i)$, for $i\in\{1,\ldots,k\}$, are maximizers of $W$ on $G$ as sets?
\end{problem}

Let us consider other \emph{interactions} between vertices of the same color. For example, given a function $g:\{1,\ldots,\diam(G)\}\to\mathbb{R}$, the second and third author defined the \emph{$g$-energy} \cite{BCL} of a set $A$ of vertices of a graph $G$ as the quantity 
\[E_g(A)=\sum_{\{u,v\}\in\binom{A}{2}}g(d_G(u,v)).\]
Given a surjective $k$-coloring $f:V(G)\to\{1,\ldots,k\}$, we can then define the $g$-energy of $f$ to be the sum of the $g$-energies of the color classes of $f$; that is,
\[E_g(f)=\sum_{i=1}^k E_g(f^{-1}(i)).\]
We say that $f$ is a \emph{$k$-color weak minimizer of $E_g$ on $G$} if $E_g(f)\leq E_g(f')$ for all surjective $f':V(G)\to\{1,\ldots,k\}$ with $\type(f')=\type(f)$. Similarly, one can define the notion of \emph{$k$-color weak maximizer of $E_g$ on $G$}.
\begin{problem}
Given a graph $G$, and a function $g:V(G)\to\{1,\ldots,k\}$, what are the $k$-color weak minimizers (maximizers) of $E_g$ on $G$?
\end{problem}
When the function $g$ is defined by $g(x)=\frac{1}{x}$, the quantity $E_g(A)$ can be thought of as the \emph{electric potential energy} of a system of unit point charges in the finite metric space $(V(G),d_G)$, located at the vertices in $A$ (see \cite{BCL}). So, sets of vertices $A$ that minimize $E_g$ will be ``spread out as much as possible'' within $G$. The sets of vertices $A$ of $C_n$ that minimize $E_g$ are precisely the \emph{maximally even sets} introduced in Clough and Douthett's work on music theory \cite{BCL, CloughDouthett, MR2408358}.
\begin{question}\label{question_electric}
When $g(x)=\frac{1}{x}$, which surjective functions $f:V(C_n)\to\{1,\ldots,k\}$ are $k$-color weak maximizers of $E_g$ on $C_n$?
\end{question}
Since a set of vertices $A$ of a distance degree regular graph is a minimizer of $E_g$ if and only its complement is also a minimizer of $E_g$, the answer to Question \ref{question_electric} when $k=2$ is already well understood \cite{BCL}. But, very little is known regarding Question \ref{question_electric} for $k\geq 3$ (see \cite{MR4550703} for a discussion of some relevant difficulties). 

Let us end with a few more specific questions regarding the Wiener index of vertex colorings suggested by the current work that we were not able to answer.

\begin{question}
What are the $k$-color \emph{local} weak maximizers of $W$ on cycles?
\end{question}


\begin{question}
For what graphs $G$ is it the case that whenever $f,f':V(G)\to\{1,\ldots,k\}$ are $k$-color weak maximizers such that $\type(f)\maj\type(f')$ and $\type(f)\neq\type(f')$, then $W(f)<W(f')$?
\end{question}

\begin{question} For what graphs $G$ is it the case that for $f:V(G)\to\{1,\ldots,k\}$ a $k$-color weak maximizer, $W(f)=\min\{W(f')\st f'\in\WM_k(F)\}$ if and only if the number of vertices in any two color classes of $f$ differ by at most one? For what $G$ will $W(f)=\max\{W(f')\st f'\in\WM_k(G)\}$ if and only if $f$ has a color class of the largest possible size $|V(G)|-(k-1)$? \end{question}

\end{document}